\documentclass[preprint,11pt,3p]{elsarticle}

\usepackage{amssymb}
\usepackage{amsmath,amsthm}
\usepackage{mathrsfs}
\usepackage{hyperref}
\usepackage{cleveref}
\usepackage[all,cmtip]{xy}
\usepackage{epsf, graphicx}
\usepackage{latexsym,amsfonts,amsbsy,amssymb}
\usepackage{geometry}
\usepackage{titletoc}
\usepackage{color,xcolor}
\usepackage{rotating}
\usepackage{algorithm}
\usepackage{algorithmic}
\usepackage{amscd}
\makeatletter

\@addtoreset{equation}{section} \makeatother
\newtheorem{theorem}{Theorem}[section]
\newtheorem{lemma}[theorem]{Lemma}
\newtheorem{notation}[theorem]{Notation}
\newtheorem{conjecture}[theorem]{Conjecture}
\newtheorem{corollary}[theorem]{Corollary}%[section]
\newtheorem{remark}[theorem]{Remark}%[section]
\newtheorem{void}[theorem]{}
\newtheorem{proposition}[theorem]{Proposition}%[section]
\def\Gal{{\rm Gal}}
\def\Q{\mathbb{Q}}
\def\Z{\mathbb{Z}}

\makeatletter
\def\ps@pprintTitle{%
	\let\@oddhead\@empty
	\let\@evenhead\@empty
	\def\@oddfoot{\reset@font\hfil\thepage\hfil}
	\let\@evenfoot\@oddfoot
}
\makeatother

\begin{document}

\begin{frontmatter}
	
\title{The Blockwise Navarro Alperin Weight Conjecture for Double Covers of Symmetric and Alternating Groups}
	
\author[label1]{Yucong Du}
\ead{duyc@cqu.edu.cn}
	
\author[label2]{Xin Huang}%\fnref{label3}}  %\corref{cor1}  correspondence auth.
\ead{xinhuang@mails.ccnu.edu.cn}

\author[label3]{Jiping Zhang}%\fnref{label3}}  %\corref{cor1}  correspondence auth.
\ead{jzhang@math.pku.edu.cn}

\address[label1]{College of Mathematics and Statistics, Chongqing University, Chongqing 401331, China}
\address[label2]{School of Mathematics and Statistics, Central China Normal University, Wuhan 430079, China}
\address[label3]{School of Mathematical Sciences, Peking University, Beijing 100871, China}

\begin{abstract}
We prove the blockwise Navarro Alperin weight conjecture for double covers of symmetric and alternating groups.
\end{abstract}


\end{frontmatter}

\section{Introduction}

There are two double covering groups of $S_n$, denoted by $\tilde{S}_n^+$ and $\tilde{S}_n^-$, such that
$$
\tilde{S}_n^\eta=\left\langle s_1,\ldots,s_{n-1}\mid s_i^2=\eta,(s_is_{i+1})^3=\eta,[s_i,s_j]=-1\mbox{ if }|i-j|\geq2\right\rangle,
$$
where $\eta\in\{\pm1\}$ and $-1$ represents the central element in $\tilde{S}_n^\eta$ of order $2$. Let $\varepsilon\colon\tilde{S}_n^\eta\to\{\pm1\}$ be the sign character (that is, $\varepsilon$ is a group homomorphism sending every $s_i$ to $-1$) and $\tilde{A}_n=\ker(\varepsilon)$ be the double cover of $A_n$. 

Let $p$ be a prime number.  Denote by $\mathbb{Q}_p$ the field of $p$-adic numbers and fix an algebraic closure $\bar{\mathbb{Q}}_p$ of $\mathbb{Q}_p$. Let $G$ be a finite group. We consider ordinary (resp. $p$-Brauer) characters of $G$ as functions from the set of elements (resp. $p$-regular elements) of $G$ to $\bar{\mathbb{Q}}_p$. Denote by $\mathrm{Irr}(G)$ (resp. $\mathrm{IBr}(G)$) the set of irreducible ordinary (resp. $p$-Brauer) characters of $G$. An irreducible ordinary (resp. $p$-Brauer) character of $\tilde{S}_n^\eta$ or $\tilde{A}_n$ is called a spin character (resp. $p$-Brauer character) if the element $-1$ is not in its kernel. The corresponding representation is called a spin representation.

Denote by $\mathcal{W}(G)$ the set of conjugacy classes of pairs $(R,\varphi)$, where $R$ is a $p$-subgroup of $G$ and $\varphi$ is a $p$-defect-zero character of $N_G(R)/R$, and $(R,\varphi)$ is called a $p$-weight of $G$. Note that if $N_G(R)/R$ has a $p$-defect-zero character, then $R$ is a $p$-radical subgroup of $G$, i.e., $O_p(N_G(R))=R$. The Galois group $\mathrm{Gal}(\bar{\mathbb{Q}}_p/\mathbb{Q}_p)$ acts on both $\mathrm{Irr}(G)$ and $\mathrm{IBr}(G)$, hence $\mathrm{Gal}(\bar{\mathbb{Q}}_p/\mathbb{Q}_p)$  also acts on $\mathcal{W}(G)$ by acting on the second component of a weight. It is also clear that the outer automorphism group $\mathrm{Out}(G)$ acts on $\mathrm{Irr}(G)$, $\mathrm{IBr}(G)$ and $\mathcal{W}(G)$. In \cite{Nav}, Navarro proposed the blockwise Navarro Alperin weight conjecture (which is also known as the blockwise Galois Alperin weight conjecture), stated as follows.

\begin{conjecture}[{\cite{Nav}}, {\cite{Tur}}]\label{conjecture:Galois-Alperin}
There is an $(\mathrm{Out}(G)\times\mathrm{Gal}(\bar{\mathbb{Q}}_p/\mathbb{Q}_p))$-equivariant bijection
$$
\Omega\colon\mathrm{IBr}(G)\to\mathcal{W}(G)
$$
such that $\mathrm{bl}(\phi)=\mathrm{bl}(\Omega(\phi))^G$ for any $\phi\in \mathrm{IBr}(G)$.
\end{conjecture}

\begin{remark}\label{remark:on Galois groups}
{\rm Let $t$ be a positive integer which is not divisible by $p$, and let $\omega\in \bar{\mathbb{Q}}_p$ be a primitive $t$-th root of unity.  The restriction of $\mathrm{Gal}(\bar{\mathbb{Q}}_p/\mathbb{Q}_p)$ to $\mathrm{Gal}(\mathbb{Q}(\omega)/\mathbb{Q})$ yields the subgroup of $\mathrm{Gal}(\mathbb{Q}(\omega)/\mathbb{Q})$ generated by the automorphism $\sigma$ of $\mathbb{Q}(\omega)$, sending $\omega$ to $\omega^p$. We denote the group $\langle \sigma\rangle$ by $\Gamma_t$. Note that the group $\Gamma_t$ is independent of the choice of a primitive root of unity $\omega$. If $t'$ is another positive integer which is a multiple of $t$, then the restriction map defines a surjective group homomorphism $\Gamma_{t'}\to \Gamma_t$. Hence the $\Omega$ in Conjecture \ref{conjecture:Galois-Alperin} is $\mathrm{Gal}(\bar{\mathbb{Q}}_p/\mathbb{Q}_p)$-equivariant if and only if $\Omega$ is $\Gamma_{|G|_{p'}}$-equivariant (where $|G|_{p'}$ denotes the $p'$-part of the integer $|G|$).}   
\end{remark}

The blockwise Navarro Alperin weight conjecture has been proved for $p$-solvable groups (see \cite{Tur}), but few results are known for simple groups. Let $n$ be a positive integer. In this paper we prove the following result.

\begin{theorem}\label{main}
If $p$ is an odd prime, the blockwise Navarro Alperin weight conjecture holds for spin blocks of $\tilde{S}_n^+$, $\tilde{S}_n^-$ and $\tilde{A}_n$.
\end{theorem}

In \cite{DH}, Conjecture \ref{conjecture:Galois-Alperin} has been proved for symmetric groups and alternating groups. Moreover, since there is a natural bijection between irreducible $2$-Brauer characters and $2$-weights of $\tilde{S}_n^+$, $\tilde{S}_n^-$ or $\tilde{A}_n$ and that of $S_n$ or $A_n$, combining with Theorem~\ref{main}, we immediately get the following result.

\begin{theorem}
The blockwise Navarro Alperin weight conjecture holds for $\tilde{S}_n^+$, $\tilde{S}_n^-$ and $\tilde{A}_n$ and any prime $p$.
\end{theorem}

This article is organized as follows: In Section \ref{Humphreys_product}, we provide the information of the twisted central product and the Humphreys product. In Section \ref{section3}, we determine the action of $\mathrm{Gal}(\bar{\mathbb{Q}}_p/\mathbb{Q}_p)$ on spin characters of $\tilde{S}_n^\eta$. In Section \ref{section4}, we provide a parametrization of irreducible $p$-Brauer characters of $\tilde{S}_n^\eta$ and $\tilde{A}_n$ in spin blocks, and the action of $\mathrm{Gal}(\bar{\mathbb{Q}}_p/\mathbb{Q}_p)$ is determined. In Section \ref{section5}, we present the radical subgroups of $\tilde{S}_n^\eta$. In Section \ref{section6}, we present the construction of $p$-defect-zero representations for each component of Humphreys products in the normalizers of $p$-radical subgroups. In Section \ref{section7}, we use the explicit construction to determine the action of $\mathrm{Gal}(\bar{\mathbb{Q}}_p/\mathbb{Q}_p)$ on weights of $\tilde{S}_n^\eta$ or $\tilde{A}_n$. In Section \ref{section8}, we complete the proof of Theorem~\ref{main}.

\section{The Humphreys Product}\label{Humphreys_product}

Throughout the rest of this paper we assume that $p$ is an odd prime. If $\tilde{H}_1,\tilde{H}_2,\ldots,\tilde{H}_m$ are subgroups of $\tilde{S}_n^\eta$ whose images in $S_n$ operate on disjoint sets and $\tilde{H}=\tilde{H}_1\tilde{H}_2\cdots\tilde{H}_m$, then we write
$$
\tilde{H}=\tilde{H}_1*\tilde{H}_2*\cdots*\tilde{H}_m
$$
and call $\tilde{H}$ the twisted central product of $\tilde{H}_1,\tilde{H}_2,\ldots,\tilde{H}_m$. Then for any $h_i\in\tilde{H}_i$ and $h_j\in\tilde{H}_j$ with $i\neq j$, we have
$$
[h_i,h_j]=\left\{\begin{array}{rl}
-1,&\mbox{if }\varepsilon(h_i)=\varepsilon(h_j)=-1,\\
1,&\mbox{else},
\end{array}\right.
$$
where $\varepsilon$ is the sign character of $\tilde{S}_n^\eta$. By restriction we view $\varepsilon$ as a sign character of each $\tilde{H}_i$.

Let $\chi_1,\chi_2,\ldots,\chi_r$ be self-associated irreducible spin ordinary (resp. $p$-Brauer) characters of $\tilde{H}_1,\tilde{H}_2,\ldots,\tilde{H}_r$ (i.e., $\varepsilon\chi_i=\chi_i$ for $1\leq i\leq r$) and let $\chi_{r+1},\chi_{r+2},\ldots,\chi_m$ be non-self-associated irreducible spin characters of $\tilde{H}_{r+1},\tilde{H}_{r+2},\ldots,\tilde{H}_m$ (i.e., $\varepsilon\chi_i\neq\chi_i$ for $r+1\leq i\leq m$). Here we abusively use the same notation $\varepsilon$ to denote both ordinary and $p$-Brauer sign characters. Set $s=m-r$. For $1\leq i\leq r$, assume that
$$
\mathrm{Res}_{\tilde{H}_i\cap\tilde{A}_n}^{\tilde{H}_i}\chi_i=\chi_i^++\chi_i^-
$$
for two irreducible ordinary (resp. $p$-Brauer) characters $\chi_i^+$ and $\chi_i^-$ of $\tilde{H}_i\cap\tilde{A}_n$. Denote $\delta_i=\chi_i^+-\chi_i^-$.

\begin{theorem}[Humphreys Product \cite{Hum}]\label{symbol_of_Humphreys_product}
		%	Let $\chi_j$ be an irreducible ordinary (resp. $p$-Brauer) character of $\tilde{H}_j$ for all $1\leq j\leq m$ with $p$ an odd prime, and 
Keep the notation above. We can construct an irreducible ordinary (resp. $p$-Brauer) character of $\tilde{H}$, denoted by
$$
\chi=\chi_1*\chi_2*\cdots*\chi_m\in\mathrm{Irr}(\tilde{H}_1*\tilde{H}_2*\cdots*\tilde{H}_m)
$$
(resp. $\chi=\chi_1*\chi_2*\cdots*\chi_m\in\mathrm{IBr}(\tilde{H}_1*\tilde{H}_2*\cdots*\tilde{H}_m))$, such that $\chi$ covers $\chi_i$ for each $1\leq i\leq m$, and $\chi$ and $\varepsilon\chi$ are the only irreducible characters satisfy this condition. This character $\chi$ is called the Humphreys product of $\chi_1,\chi_2,\ldots,\chi_m$. Let $t_i\in\tilde{H}_i\cap\tilde{A}_n$ if $1\leq i\leq r$ and $t_i\in\tilde{H}_i-\tilde{A}_n$ if $r+1\leq i\leq m$. The following hold.
\begin{enumerate}
\item If $s$ is odd, then $\varepsilon\chi\neq\chi$ and
$$
\chi(t_1t_2\cdots t_m)=(2\mathrm{i})^{\lfloor s/2\rfloor}\cdot\prod_{i=1}^r\delta_i(t_i)\cdot\prod_{i=r+1}^m\chi_i(t_i),
$$
where $\mathrm{i}$ is a primitive $4$-th root of unity in $\bar{\mathbb{Q}}_p$, and the notation $\lfloor s/2 \rfloor$ means the largest integer which is not greater than $s/2$.
\item If $s$ is even, then $\varepsilon\chi=\chi$ and
$$
\delta(t_1t_2\cdots t_m)=(2\mathrm{i})^{\lfloor s/2\rfloor}\cdot\prod_{i=1}^r\delta_i(t_i)\cdot\prod_{i=r+1}^m\chi_i(t_i),
$$
where $\delta=\chi^+-\chi^-$ and $\mathrm{Res}_{\tilde{H}\cap\tilde{A}_n}^{\tilde{H}}\chi=\chi^++\chi^-$.
\end{enumerate}
Moreover, the degree of $\chi$ is
$$
\chi(1)=2^{\lfloor s/2\rfloor}\cdot\prod_{i=1}^m\chi_i(1).
$$
So $\chi$ is a $p$-defect-zero irreducible character if and only if all $\chi_i$'s are $p$-defect-zero irreducible characters.
\end{theorem}

This result can also be found in \cite[Lemma~4.2, Proposition~4.4]{HH}. We point out that the irreducible $p$-Brauer characters have the same properties according to Humphreys' construction. For example, let $m=s=2$ and $P_1,P_2$ be irreducible representations corresponding to irreducible $p$-Brauer characters $\chi_1,\chi_2$, respectively, and set $S(\gamma_1,\gamma_2)=P_1(\gamma_1)\otimes P_2(\gamma_2)$ for $\gamma_i\in\tilde{H}_i$. Then according to the proof of \cite[Theorem~2.4]{Hum}, the representation of $\mathrm{Res}_{\tilde{H}\cap\tilde{A}_n}^{\tilde{H}}(\chi_1*\chi_2)$ is defined by
$$
R(\gamma_1,\gamma_2)=\left\{\begin{array}{rl}
\left(\begin{array}{rr}
1&0\\
0&1
\end{array}\right)\otimes S(\gamma_1,\gamma_2),&\mbox{if }\varepsilon(\gamma_1)=\varepsilon(\gamma_2)=1,\\
\left(\begin{array}{rr}
0&-1\\
1&0
\end{array}\right)\otimes S(\gamma_1,\gamma_2),&\mbox{if }\varepsilon(\gamma_1)=\varepsilon(\gamma_2)=-1.
\end{array}\right.
$$
So we have $\delta(\gamma_1\gamma_2)=2\mathrm{i}\cdot\chi_1(\gamma_1)\chi_2(\gamma_2)$ if $\varepsilon(\gamma_1)=\varepsilon(\gamma_2)=-1$.

\section{The Irreducible Characters of Double Covers of Symmetric Groups}\label{section3}

%	Let $\sigma$ be the field automorphism of $\bar{\mathbb{Q}}_p$ such that $\sigma$ fixes roots of unity of $p$-power and $\sigma(\omega)=\omega^p$ for any $p'$-th root of unity $\omega\in\bar{\mathbb{Q}}_p$.

Let $n_0$ be an integer such that for every finite group $G$ in the rest of this paper, $n_0$ is divisible by $|G|$. Assume further that $4~|~n_0$. Let $\omega$ be a primitive $n_0$-th root of unity in $\bar{\mathbb{Q}}_p$. We can uniquely write $\omega=\omega_1\omega_2$ such that $\omega_1$ is a $p'$-th root of unity and $\omega_2$ is a $p$-power-th root of unity.  Let $\sigma\in \Gal(\Q(\omega)/\Q)$ be the Galois automorphism sending $\omega_1$ to $\omega_1^p$ and fixes $\omega_2$. Then the restriction yields a surjective homomorphism $\langle\sigma\rangle\to\Gamma_{|G|_{p'}}$ (see Remark \ref{remark:on Galois groups} for the notation $\Gamma_{|G|_{p'}}$). Hence the map $\Omega$ for the group $G$ in Conjecture \ref{conjecture:Galois-Alperin} is $\Gamma_{|G|_{p'}}$-equivariant if and only if $\Omega$ is $\langle \sigma\rangle$-equivariant.

%	Let $q$ be an integer which is not divisible by $p$, and let $\omega$ be any primitive $m$-th root of unity such that $p\nmid m$ and that $\{\mathrm{i},\sqrt{p},\sqrt{q}\}\subseteq\mathbb{Q}(\omega)$. Let $\sigma$ be the automorphism of $\mathbb{Q}(\omega)$ sending $\omega$ to $\omega^p$.

\begin{proposition}
We have $\sigma(\sqrt{q})=\left(\frac{q}{p}\right)\sqrt{q}$ for any integer $q$ not divisable by $p$ satisfying $\sqrt{q}\in\mathbb{Q}(\omega)$ and $\sigma(\sqrt{-p})=\sqrt{-p}$. Here $\left(\frac{\cdot}{p}\right)$ is the Legendre symbol. 
\end{proposition}

\begin{proof}
\begin{enumerate}
\item Since $\sqrt{q}$ is an algebraic integer, we have $\sqrt{q}\in\mathbb{Z}[\omega]$. Let $\nu_p$ be the $p$-valuation of $\mathbb{Z}[\omega]$ and $\pi\colon\mathbb{Z}[\omega]\to\mathbb{F}_p(\bar{\omega})$ be the canonical homomorphism to its residue field such that $\bar{\omega}=\pi(\omega)$. Let $\bar{\sigma}$ be the field automorphism of $\mathbb{F}_p(\bar{\omega})$ defined by $\sigma$, that is, $\bar{\sigma}\circ\pi=\pi\circ\sigma$. Then $\bar{\sigma}(\bar{\omega})=\bar{\omega}^p$ is the Frobenius automorphism of $\mathbb{F}_p(\bar{\omega})$. 
\begin{enumerate}
\item If $q$ is a quadratic residue modulo $p$, then $\pi(\sqrt{q})\in\mathbb{F}_p$ and hence $\bar{\sigma}(\pi(\sqrt{q}))=\pi(\sqrt{q})$;
\item If $q$ is not a quadratic residue modulo $p$, then $\pi(\sqrt{q})\notin\mathbb{F}_p$ and hence $\bar{\sigma}(\pi(\sqrt{q}))\neq\pi(\sqrt{q})$, so we have $\bar{\sigma}(\pi(\sqrt{q}))=-\pi(\sqrt{q})$ since $\bar{\sigma}(\pi(\sqrt{q}))^2=\pi(q)=\pi(\sqrt{q})^2$.
\end{enumerate}
As a conclusion, we have $\pi(\sigma(\sqrt{q}))=\bar{\sigma}(\pi(\sqrt{q}))=\left(\frac{q}{p}\right)\pi(\sqrt{q})$. Since $\sigma(\sqrt{q})^2=q=(\sqrt{q})^2$, we know that $\sigma(\sqrt{q})\in\{\pm\sqrt{q}\}$, so $\sigma(\sqrt{q})=\left(\frac{q}{p}\right)\sqrt{q}$.
\item Let $\omega_p$ be a $p$-th primitive root of unity in $\mathbb{Q}(\omega)$. Then $\sigma(\omega_p)=\omega_p$. 
\begin{enumerate}
\item If $p\equiv1\mod4$, then we have
$$
\sqrt{p}=\sum_{t=1}^{p-1}\left(\frac{t}{p}\right)\omega_p^t;
$$
see the proof of \cite[Proposition 4.1]{BN}. Hence $\sigma(\sqrt{p})=\sqrt{p}$. Moreover, since $\sigma(\sqrt{-1})=\left(\frac{-1}{p}\right)\sqrt{-1}=\sqrt{-1}$, we have $\sigma(\sqrt{-p})=\sqrt{-p}$.
\item If $p\equiv3\mod4$, then we have
$$
\mathrm{i}\sqrt{p}=\sqrt{-p}=\sum_{t=1}^{p-1}\left(\frac{t}{p}\right)\omega_p^t;
$$
see the proof of \cite[Proposition 4.1]{BN}. Hence $\sigma(\sqrt{-p})=\sqrt{-p}$.
\end{enumerate}
\end{enumerate}
\end{proof}

\begin{remark}\label{remark:action of sigma}
{\rm	Note that for any integer $m\neq0$, we have a decomposition
$$
m=(-p)^{\nu_p(m)}\cdot(-1)^{\nu_p(m)}m_{p'}, 
$$
where $\nu_p$ is the $p$-valuation and $m_{p'}$ is the $p'$-part of $m$. Taking $q=(-1)^{\nu_p(m)}m_{p'}$ in the above proposition, we obtain
$$
\sigma(\sqrt{m})=\left(\frac{(-1)^{\nu_p(m)}m_{p'}}{p}\right)\sqrt{m}.
$$
}
\end{remark}

\begin{void}	
{\rm \textbf{Bar partitions.}	The irreducible spin characters of $\tilde{S}_n^\eta$ are originally determined in \cite{Sch}, which are parametrized by strict partitions (which are exactly the bar partitions in \cite{Ols93}) of $n$; see \cite[Section~7]{Ols93}. We restate these results. If $\lambda=(n_1,n_2,\ldots,n_l)$ is a partition of $n$, then we write $|\lambda|=n$, and call it a bar partition if $n_1>n_2>\cdots>n_l>0$.
\begin{enumerate}
\item If $n-l$ is odd, then $\lambda$ defines two non-self-associated characters $\chi_\lambda$ and $\varepsilon\chi_\lambda$ of $\tilde{S}_n^\eta$, and in this case we write $\lambda\in\mathcal{P}_{n,\textrm{odd}}^{\textrm{str}}$ and call $\lambda$ non-self-associated;
\item If $n-l$ is even, then $\lambda$ defines one self-associated character $\chi_\lambda$ of $\tilde{S}_n^\eta$, and in this case we write $\lambda\in\mathcal{P}_{n,\textrm{even}}^{\textrm{str}}$ and call $\lambda$ self-associated.
\end{enumerate}
Note that we have
$$
n-l\equiv\#\{\mbox{even numbers in }\lambda\}\mod2.
$$
Set
$$
n'_j=\left\{\begin{array}{ll}
(-1)^{\nu_p(n_j)}\cdot(n_j)_{p'},&\mbox{if }n_j\mbox{ is odd},\\
(-1)^{\nu_p(n_j)}\cdot(n_j)_{p'}/(2\eta),&\mbox{if }n_j\mbox{ is even}
\end{array}\right.
$$
for any $1\leq j\leq l$ and define
$$
N_\lambda^\eta=(-1)^{\lfloor(n-l)/2\rfloor}\prod_{j=1}^ln'_j.
$$
}
\end{void}

\begin{void}
{\rm		\textbf{Irreducible Characters of $\tilde{S}_n^+$.}	Let $\eta=+1$ and recall that
$$
\tilde{S}_n^+=\left\langle s_1,\ldots,s_{n-1}\mid s_i^2=1,(s_is_{i+1})^3=1,[s_i,s_j]=-1\mbox{ if }|i-j|\geq2\right\rangle.
$$
The construction of irreducible ordinary characters of $\tilde{S}_n^+$ can be found in \cite[Section~8]{HH}. First we can define the Schur Hauptdarstellung of $\tilde{S}_n^+$, which is an irreducible ordinary representation of $\tilde{S}_n^+$ whose degree is a power of $2$, denoted by $\Delta'_n$. Let $\chi_{\Delta'_n}$ be the irreducible character of $\Delta'_n$. It is non-self-associated if $n$ is even, and self-associated if $n$ is odd. For an odd number $n$, if
$$
\mathrm{Res}_{\tilde{A}_n}^{\tilde{S}_n^+}\chi_{\Delta'_n}=\chi'^+_{\Delta'_n}+\chi'^-_{\Delta'_n},
$$
we set $\delta_{\Delta'_n}=\chi'^+_{\Delta'_n}-\chi'^-_{\Delta'_n}$.
\begin{enumerate}
\item For any $s\in\tilde{A}_n$, by \cite[Theorem~8.4~(i)]{HH}, we have $\chi_{\Delta'_n}(s)\in\mathbb{Z}$ (more specifically $\chi_{\Delta'_n}(s)\in\{0,\pm2^m\mid m\in\mathbb{Z}\}$). Hence $\sigma$ fixes ${\rm Res}_{\tilde{A}_n}^{\tilde{S}_n}\chi_{\Delta'_n}$. Then by the Gallagher correspondence (see e.g. \cite[Corollary 1.23]{Navbook}), we have $\sigma(\{\chi_{\Delta'_n},\varepsilon\chi_{\Delta'_n}\})=\{\chi_{\Delta'_n},\varepsilon\chi_{\Delta'_n}\}$.
\item If $n$ is even and let $s=s_1s_2\cdots s_{n-1}\in\tilde{S}_n^+-\tilde{A}_n$, then by \cite[Theorem~8.4~(ii)]{HH} we have
$$
\chi_{\Delta'_n}(s)=\mathrm{i}^{n/2-1}\sqrt{n/2}.
$$
So by Remark \ref{remark:action of sigma},
\begin{equation}\label{symbol_nsa}
\sigma(\chi_{\Delta'_n}(s))=\left(\frac{-1}{p}\right)^{n/2-1}\left(\frac{(-1)^{\nu_p(n)}n_{p'}/2}{p}\right)\cdot\chi_{\Delta'_n}(s).
\end{equation}
\item If $n$ is odd and let $s=s_1s_2\cdots s_{n-1}\in\tilde{A}_n$, then by \cite[Theorem~8.4~(iii)]{HH} we have
$$
\delta_{\Delta'_n}(s)=\mathrm{i}^{(n-1)/2}\sqrt{n},
$$
Again by Remark \ref{remark:action of sigma}, we have
\begin{equation}\label{symbol_sa}
\sigma(\delta_{\Delta'_n}(s))=\left(\frac{-1}{p}\right)^{(n-1)/2}\left(\frac{(-1)^{\nu_p(n)}n_{p'}}{p}\right)\cdot\delta_{\Delta'_n}(s).
\end{equation}
\end{enumerate}

According to \cite[Page~1221]{HH}, for a bar partition $\lambda=(n_1>n_2>\cdots>n_l>0)$, we can define an irreducible character $\chi_\lambda$ as follows: Set
$$
\tilde{S}_\lambda^+=\tilde{S}_{n_1}^+*\tilde{S}_{n_2}^+*\cdots*\tilde{S}_{n_l}^+
$$
and let
$$
\tilde{\chi}_\lambda=\mathrm{Ind}_{\tilde{S}_\lambda^+}^{\tilde{S}_n^+}\left(\chi_{\Delta'_{n_1}}*\chi_{\Delta'_{n_2}}*\cdots*\chi_{\Delta'_{n_l}}\right).
$$
Then the irreducible character $\chi_\lambda\in\mathrm{Irr}(\tilde{S}_n^+)$ is defined such that $\chi_{(n)}=\chi_{\Delta'_n}$ and
\begin{equation}\label{decomposition_number}
\tilde{\chi}_\lambda=\chi_\lambda+\sum_{\lambda'<\lambda,\lambda'\in\mathcal{P}_{n,\textrm{even}}^{\textrm{str}}}m(\lambda',\lambda)\chi_{\lambda'}+\sum_{\lambda'<\lambda,\lambda'\in\mathcal{P}_{n,\textrm{odd}}^{\textrm{str}}}m(\lambda',\lambda)(\chi_{\lambda'}+\varepsilon\chi_{\lambda'}),
\end{equation}
where $m(\lambda',\lambda)$ is a non-negative integer depends on $\lambda'$ and $\lambda$. Recall that in this case we have
\begin{equation*}
N_\lambda^+=(-1)^{\lfloor(n-l)/2\rfloor}\cdot\prod_{n_j\mbox{ odd}}(-1)^{\nu_p(n_j)}(n_j)_{p'}\cdot\prod_{n_j\mbox{ even}}\frac{(-1)^{\nu_p(n_j)}(n_j)_{p'}}{2}.
\end{equation*}
}
\end{void}

\begin{theorem}\label{symbol_of_irr_+}
Let $\lambda=(n_1>n_2>\cdots>n_l>0)$ be a bar partition of $n$ and $\chi_\lambda$ be the corresponding irreducible character of $\tilde{S}_n^+$. Set $\mu_\lambda=\left(\frac{N_\lambda^+}{p}\right)$.
\begin{enumerate}
\item If $n-l$ is odd, then setting $\chi_\lambda^+=\chi_\lambda$, $\chi_\lambda^-=\varepsilon\chi_\lambda$, we have $\sigma(\chi_\lambda^+)=\chi_\lambda^{\mu_\lambda}$;
\item If $n-l$ is even, then setting $\mathrm{Res}_{\tilde{A}_n}^{\tilde{S}_n^+}\chi_\lambda=\chi'^+_\lambda+\chi'^-_\lambda$, we have $\sigma(\chi'^+_\lambda)=\chi'^{\mu_\lambda}_\lambda$.
\end{enumerate}
\end{theorem}

\begin{proof}
Since $\sigma(\{\chi_{\Delta'_{n_j}},\varepsilon\chi_{\Delta'_{n_j}}\})=\{\chi_{\Delta'_{n_j}},\varepsilon\chi_{\Delta'_{n_j}}\}$, we have $\sigma(\tilde{\chi}_\lambda)=\tilde{\chi}_\lambda$ or $\varepsilon\tilde{\chi}_\lambda$. By direct calculation
\begin{equation*}
\begin{split}
\mu_\lambda=\left(\frac{-1}{p}\right)^{\lfloor s/2\rfloor}&\cdot\prod_{n_j\mbox{ odd}}\left(\frac{-1}{p}\right)^{(n_j-1)/2}\left(\frac{(-1)^{\nu_p(n_j)}(n_j)_{p'}}{p}\right)\\
&\cdot\prod_{n_j\mbox{ even}}\left(\frac{-1}{p}\right)^{n_j/2-1}\left(\frac{(-1)^{\nu_p(n_j)}(n_j)_{p'}/2}{p}\right),
\end{split}
\end{equation*}
where $s=\#\{\mbox{even numbers in }\lambda\}$. Now we apply Theorem~\ref{symbol_of_Humphreys_product} and obtain the following:
\begin{enumerate}
\item If $n-l$ is odd, then setting $\tilde{\chi}_\lambda^+=\tilde{\chi}_\lambda$, $\tilde{\chi}_\lambda^-=\varepsilon\tilde{\chi}_\lambda$, we have $\sigma(\tilde{\chi}_\lambda^+)=\tilde{\chi}_\lambda^{\mu_\lambda}$ by the structure of $\tilde{\chi}_\lambda$ and the equality~(\ref{symbol_nsa}). Since
$$
\tilde{\chi}_\lambda-\chi_\lambda=\varepsilon\tilde{\chi}_\lambda-\varepsilon\chi_\lambda=\sigma(\tilde{\chi}_\lambda)-\sigma(\chi_\lambda)
$$
by the equality~(\ref{decomposition_number}), we get $\sigma(\chi_\lambda^+)=\chi_\lambda^{\mu_\lambda}$.
\item If $n-l$ is even, then setting $\mathrm{Res}_{\tilde{A}_n}^{\tilde{S}_n^+}\tilde{\chi}_\lambda=\tilde{\chi}'^+_\lambda+\tilde{\chi}'^-_\lambda$, we have $\sigma(\tilde{\chi}'^+_\lambda)=\tilde{\chi}'^{\mu_\lambda}_\lambda$  by the structure of $\tilde{\chi}_\lambda$ and the equality~(\ref{symbol_sa}). By the equality~(\ref{decomposition_number}), we have
$$
\tilde{\chi}'^+_\lambda+\tilde{\chi}'^-_\lambda=\chi'^+_\lambda+\chi'^-_\lambda+\sum_j\phi_j,
$$
where $\phi_j$ is an irreducible component of $\mathrm{Res}_{\tilde{A}_n}^{\tilde{S}_n^+}\chi_{\lambda'}$ for some $\lambda'<\lambda$. Hence by relabeling of $\tilde{\chi}'^+_\lambda$ and $\tilde{\chi}'^-_\lambda$ we may assume that
$$
\tilde{\chi}'^+_\lambda=\chi'^+_\lambda+\sum_j\phi'_j\mbox{ and }\tilde{\chi}'^-_\lambda=\chi'^-_\lambda+\sum_j\phi''_j,
$$
where $\phi'_j$ (resp. $\phi''_j$) is an irreducible component of $\mathrm{Res}_{\tilde{A}_n}^{\tilde{S}_n^+}\chi_{\lambda'}$ for some $\lambda'<\lambda$. Thus we have $\sigma(\chi'^+_\lambda)=\chi'^{\mu_\lambda}_\lambda$.
\end{enumerate}
\end{proof}

\begin{void}
{\rm 
\textbf{The Irreducible Characters of $\tilde{S}_n^-$.}	Recall that
\begin{equation*}
\begin{split}
N_\lambda^-&=(-1)^{\lfloor(n-l)/2\rfloor}\cdot\prod_{n_j\mbox{ odd}}(-1)^{\nu_p(n_j)}(n_j)_{p'}\cdot\prod_{n_j\mbox{ even}}\frac{(-1)^{\nu_p(n_j)}(n_j)_{p'}}{-2}=(-1)^{n-l}N_\lambda^+.
\end{split}
\end{equation*}
If
\begin{eqnarray*}
\Gamma\colon\tilde{S}_n^+&\to&\mathrm{GL}_m(\bar{\mathbb{Q}}_p),\\
s_j&\mapsto&\Gamma(s_j)
\end{eqnarray*}
is an irreducible spin representation of $\tilde{S}_n^+$, then we can check that
\begin{eqnarray*}
\Gamma'\colon\tilde{S}_n^-&\to&\mathrm{GL}_m(\bar{\mathbb{Q}}_p),\\
s_j&\mapsto&\mathrm{i}\cdot\Gamma(s_j)
\end{eqnarray*}
is an irreducible spin representation of $\tilde{S}_n^-$ as well. So we immediately get the following result.
}
\end{void}

\begin{theorem}\label{symbol_of_irr_-}
Let $\lambda=(n_1>n_2>\cdots>n_l>0)$ be a bar partition of $n$ and $\chi_\lambda$ be the corresponding irreducible character of $\tilde{S}_n^-$. Set $\mu_\lambda=\left(\frac{N_\lambda^-}{p}\right)$.
\begin{enumerate}
\item If $n-l$ is odd, then setting $\chi_\lambda^+=\chi_\lambda$, $\chi_\lambda^-=\varepsilon\chi_\lambda$, we have $\sigma(\chi_\lambda^+)=\chi_\lambda^{\mu_\lambda}$;
\item If $n-l$ is even, then setting $\mathrm{Res}_{\tilde{A}_n}^{\tilde{S}_n^+}\chi_\lambda=\chi'^+_\lambda+\chi'^-_\lambda$, we have $\sigma(\chi'^+_\lambda)=\chi'^{\mu_\lambda}_\lambda$.
\end{enumerate}
\end{theorem}

\begin{void}\label{void:blocks of double covers}
{\rm	\textbf{Blocks of $\tilde{S}_n^\eta$.} Let $\lambda=(n_1>n_2>\cdots>n_l>0)$ be a bar partition of $n$ and assume that $p$ is an odd prime. As in \cite[Section~5~(5)]{Ols93}, we can define three types of operations on $\lambda$.
\begin{itemize}
\item Type 1: Replace $n_j$ by $n_j-p$ if $n_j>p$ and $n_j-p\notin\lambda$;
\item Type 2: Remove $p$ if $p\in\lambda$;
\item Type 3: Remove both $n_j$ and $p-n_j$ if $p-n_j\in\lambda$.
\end{itemize}
Continue these operations until no more operation is allowed. If we finally get a bar partition $\kappa$, then $\kappa$ is called the $p$-bar-core of $\lambda$. The spin blocks of $\tilde{S}_n^\eta$ are parametrized by $p$-bar-core partitions. More specifically, if $\lambda$ itself is a $p$-bar-core, then $\chi_\lambda$ is a $p$-defect-zero character and we denote this defect-zero block by $B_\lambda$. Otherwise, strict partitions $\lambda$ and $\lambda'$ have the same $p$-bar-core $\kappa$ if and only if irreducible characters $\chi_\lambda$ and $\chi_{\lambda'}$ belong to the same $p$-block, denoted by $B_\kappa$; see \cite[Theorem~(13.1)]{Ols93}. Moreover, a positive defect spin $p$-block of $\tilde{S}_n^\eta$ covers only one spin $p$-block of $\tilde{A}_n$; see \cite[Page~89]{Ols93}.
}
\end{void}

\begin{theorem}\label{symbol_of_IBr}
If $\lambda=(n_1>n_2>\cdots>n_l>0)$ is a bar partition of $n$ and $\kappa$ is the $p$-bar-core of $\lambda$. Set $n_0=|\kappa|$ and set $n=n_0+pw$. Assume that $p\nmid n_j$ for all $n_j\in\lambda$.
\begin{enumerate}
\item If $\lambda\in\mathcal{P}_{n,\textrm{even}}^{\textrm{str}}$ or $\kappa\in\mathcal{P}_{n_0,\textrm{odd}}^{\textrm{str}}$, then
$$
\left(\frac{N_\lambda^\eta}{p}\right)=\left(\frac{N_\kappa^\eta}{p}\right)\cdot\left(\frac{-1}{p}\right)^{\lfloor w/2\rfloor}\left(\frac{2\eta}{p}\right)^w;
$$
\item If $\lambda\in\mathcal{P}_{n,\textrm{odd}}^{\textrm{str}}$ or $\kappa\in\mathcal{P}_{n_0,\textrm{even}}^{\textrm{str}}$, then
$$
\left(\frac{N_\lambda^\eta}{p}\right)=\left(\frac{N_\kappa^\eta}{p}\right)\cdot\left(\frac{-1}{p}\right)^{\lceil w/2\rceil}\left(\frac{2\eta}{p}\right)^w,
$$
where the notation $\lceil w/2\rceil$ means the smallest integer which is not lower than $w/2$.
\end{enumerate}
\end{theorem}

\begin{proof}
Note that $\nu_p(n_j)=0$ and $(n_j)_{p'}=n_j$ for all $n_j\in\lambda$. Recall that
$$
N_\lambda^\eta=(-1)^{\lfloor(n-l)/2\rfloor}\prod_{j=1}^ln'_j=(-1)^{\lfloor(n-l)/2\rfloor}\cdot\prod_{n_j\mbox{ odd}}n_j\cdot\prod_{n_j\mbox{ even}}\frac{n_j}{2\eta}.
$$
Since $\lambda$ contains no element divisable by $p$, the operation of type 2 is unnecessary. Note that after one step of type 1 operation or type 3 operation, a self-associated bar partition will become a non-self-associated bar partition, and a non self-associated bar partition will become a self-associated bar partition. Note that we get the $p$-bar-core $\kappa$ by taking $w$ steps of operations on $\lambda$. It suffices to prove that
\begin{equation*}
\left(\frac{N_\lambda^\eta}{p}\right)\left(\frac{N_{\lambda'}^\eta}{p}\right)=\left\{\begin{array}{cl}
\left(\frac{2\eta}{p}\right),&\mbox{if }\lambda\mbox{ is self-associated},\\
\left(\frac{-2\eta}{p}\right),&\mbox{if }\lambda\mbox{ is non-self-associated},
\end{array}\right.
\end{equation*}
where $\lambda'$ is a strict partition obtained from $\lambda$ by applying one operation of type~1 or type 3.
\begin{itemize}
\item For an operation of type 1 at $n_{j_0}$, we have
$$
(-1)^{\lfloor(n-l)/2\rfloor}\frac{N_\lambda^\eta}{n'_{j_0}}=(-1)^{\lfloor(n-p-l)/2\rfloor}\frac{N_{\lambda'}^\eta}{(n_{j_0}-p)'}=\prod_{n_j\neq n_{j_0}}n'_j.
$$
Note that
$$
\lfloor(n-l)/2\rfloor-\lfloor(n-p-l)/2\rfloor=\left\{\begin{array}{ll}
(p+1)/2,&\mbox{if }n-l\mbox{ is even},\\
(p-1)/2,&\mbox{if }n-l\mbox{ is odd}.
\end{array}\right.
$$
We have
\begin{equation*}
\begin{split}
\left(\frac{N_\lambda^\eta}{p}\right)\left(\frac{N_{\lambda'}^\eta}{p}\right)=&\left(\frac{-1}{p}\right)^{\lfloor(n-l)/2\rfloor-\lfloor(n-p-l)/2\rfloor}\left(\frac{n'_{j_0}(n_{j_0}-p)'}{p}\right)\\
=&\left(\frac{-1}{p}\right)^{\lfloor(n-l)/2\rfloor-\lfloor(n-p-l)/2\rfloor}\left(\frac{2\eta}{p}\right)\\
=&\left(\frac{-1}{p}\right)^{(p+1)/2+(n-l)}\left(\frac{2\eta}{p}\right)=\left(\frac{-1}{p}\right)^{n-l}\left(\frac{2\eta}{p}\right).
\end{split}
\end{equation*}
\item For an operation of type 3 at $n_{j_0}$, we have
$$
(-1)^{\lfloor(n-l)/2\rfloor}\frac{N_\lambda^\eta}{n'_{j_0}(p-n_{j_0})'}=(-1)^{\lfloor(n-p-l+2)/2\rfloor}N_{\lambda'}^\eta=\prod_{n_j\neq n_{j_0},p-n_{j_0}}n'_j.
$$
Note that
$$
\lfloor(n-l)/2\rfloor-\lfloor(n-p-l+2)/2\rfloor=\left\{\begin{array}{ll}
(p-1)/2,&\mbox{if }n-l\mbox{ is even},\\
(p-3)/2,&\mbox{if }n-l\mbox{ is odd}.
\end{array}\right.
$$
We have
\begin{equation*}
\begin{split}
\left(\frac{N_\lambda^\eta}{p}\right)\left(\frac{N_{\lambda'}^\eta}{p}\right)=&\left(\frac{-1}{p}\right)^{\lfloor(n-l)/2\rfloor-\lfloor(n-p-l+2)/2\rfloor}\left(\frac{n'_{j_0}(p-n_{j_0})'}{p}\right)\\
=&\left(\frac{-1}{p}\right)^{(p-1)/2-(n-l)}\left(\frac{-2\eta}{p}\right)=\left(\frac{-1}{p}\right)^{n-l}\left(\frac{2\eta}{p}\right).
\end{split}
\end{equation*}
\end{itemize}
This completes the proof.
\end{proof}

\begin{remark}
{\rm	From the proof of Theorem \ref{symbol_of_IBr}, we can observe that $\lambda$ and $\kappa$ are both self-associated or non-self-associated if and only if $w$ is even.}
\end{remark}

\section{The Irreducible Brauer Characters of Double Covers of Symmetric Groups}\label{section4}

Let $\kappa$ be a $p$-bar-core partition. Denote by
$$
\Lambda_\kappa=\left\{\lambda=(n_1>n_2>\cdots>n_l>0)\ \middle|\ \substack{\mbox{$\kappa$ is the $p$-bar-core of $\lambda$},\\ \mbox{$p\nmid n_j$ for all $1\leq j\leq l$}}\right\}
$$
the set of bar partitions with no part divisable by $p$. In \cite{BG}, the set $\Lambda_\kappa$ is proved to define a $p$-basic set for double covers of symmetric and alternating groups.

The natural group homomorphism $\tilde{S}_n^\eta\to S_n$ provides a bijection between $p$-elements in $\tilde{S}_n^\eta$ and $p$-elements in $S_n$.

\begin{lemma}
Let $u$ be a $p$-element in $\tilde{S}_n^\eta$ whose image in $S_n$ contains $m_j$ copies of $p^j$-cycles with $0\leq j\leq l$. Then we have
$$
O_p(C_{\tilde{S}_n^\eta}(u))=I_{m_0}\times(\mathbb{Z}/p\Z)^{m_1}\times(\mathbb{Z}/p^2\Z)^{m_2}\times\cdots\times(\mathbb{Z}/p^l\Z)^{m_l},
$$
where $I_{m_0}$ denotes the trivial subgroup of $\tilde{S}_{m_0}^\eta$	and
$$
C_{\tilde{S}_n^\eta}(u)/O_p(C_{\tilde{S}_n^\eta}(u))=\tilde{S}_{m_0}^{\eta_0}*\tilde{S}_{m_1}^{\eta_1}*\tilde{S}_{m_2}^{\eta_2}*\cdots*\tilde{S}_{m_l}^{\eta_l},
$$
where $\eta_j=\left(\frac{-1}{p}\right)^j\eta$. If $\beta_j$ is an irreducible spin $p$-Brauer character of $\tilde{S}_{m_j}^{\eta_j}$ for $0\leq j\leq l$ and $\beta_0$ belongs to the spin block of $\tilde{S}_{m_0}^\eta$ defined by $\kappa$, then
$$
\mathrm{bl}(\beta_0*\beta_1*\beta_2*\cdots*\beta_l)^{\tilde{S}_n^\eta}=\mathrm{bl}(B_\kappa).
$$
\end{lemma}

\begin{proof}
Denote by $\bar{u}$ the image of $u$ in $S_n$. Since $u$ is a $2'$-element, we have $C_{\tilde{S}_n^\eta}(u)/\{\pm1\}=C_{S_n}(\bar{u})$. Since
$$
C_{S_n}(\bar{u})=S_{m_0}\times((\mathbb{Z}/p\Z)\wr S_{m_1})\times((\mathbb{Z}/p^2\Z)\wr S_{m_2})\times\cdots\times((\mathbb{Z}/p^l\Z)\wr S_{m_l}),
$$
where the reflections in $S_{m_j}$ is a product of $p^j$ copies of $2$-cycles in $S_n$, we have
$$
O_p(C_{\tilde{S}_n^\eta}(u))=I_{m_0}\times(\mathbb{Z}/p\Z)^{m_1}\times(\mathbb{Z}/p^2\Z)^{m_2}\times\cdots\times(\mathbb{Z}/p^l\Z)^{m_l}
$$
and
$$
C_{\tilde{S}_n^\eta}(u)/O_p(C_{\tilde{S}_n^\eta}(u))=\tilde{S}_{m_0}^{\eta_0}*\tilde{S}_{m_1}^{\eta_1}*\tilde{S}_{m_2}^{\eta_2}*\cdots*\tilde{S}_{m_l}^{\eta_l},
$$
where
$$
\eta_j=(-1)^{\frac{1}{2}p^j(p^j-1)}\cdot\eta^{p^l}=\left(\frac{-1}{p}\right)^j\eta.
$$
The block of $\beta_0*\beta_1*\beta_2*\cdots*\beta_l$ has been determined in \cite[Page~259]{Ols92}; from there we see that
$$
\mathrm{bl}(\beta_0*\beta_1*\beta_2*\cdots*\beta_l)^{\tilde{S}_n^\eta}=\mathrm{bl}(B_\kappa).
$$
\end{proof}

\begin{lemma}\label{lemma:4.2}
For any bar partition $\lambda$ of $n$, we set
\begin{equation*}
\begin{split}
\lambda=\left(\begin{array}{c}
n_{01}>n_{02}>\cdots>n_{0l_0}>0,\\
n_{11}p>n_{12}p>\cdots>n_{1l_1}p>0,\\
n_{21}p^2>n_{22}p^2>\cdots>n_{2l_2}p^2>0,\\
\vdots\\
n_{m1}p^m>n_{m2}p^m>\cdots>n_{ml_m}p^m>0
\end{array}\right).
\end{split}
\end{equation*}
such that $p\nmid n_{jk}$. Let
\begin{equation*}
\begin{split}
\lambda_0&=(n_{01}>n_{02}>\cdots>n_{0l_0}>0),\\
\lambda_1&=(n_{11}>n_{12}>\cdots>n_{1l_1}>0),\\
\lambda_2&=(n_{21}>n_{22}>\cdots>n_{2l_2}>0),\\
\vdots&\\
\lambda_m&=(n_{m1}>n_{m2}>\cdots>n_{ml_m}>0)
\end{split}
\end{equation*}
and $s_\lambda$ be the number of non-self-associated bar partitions in $\{\lambda_j\mid0\leq j\leq m\}$. Then we have 
$$
N_\lambda^\eta=(-1)^{\frac{1}{2}(p+1)\sum\limits_{j=0}^mjl_j}\cdot(-1)^{\lfloor s_\lambda/2\rfloor}N_{\lambda_0}^{\eta_0}N_{\lambda_1}^{\eta_1}N_{\lambda_2}^{\eta_2}\cdots N_{\lambda_m}^{\eta_m}.
$$
Recall that $\eta_j=\left(\frac{-1}{p}\right)^j\eta$.
\end{lemma}

\begin{proof}
Set $|\lambda_j|=n_{j1}+n_{j2}+\cdots+n_{jl_j}$. By definition we have
\begin{equation*}
\frac{N_\lambda^\eta}{(-1)^{\lfloor s_\lambda/2\rfloor}N_{\lambda_0}^{\eta_0}N_{\lambda_1}^{\eta_1}N_{\lambda_2}^{\eta_2}\cdots N_{\lambda_m}^{\eta_m}}=\frac{(-1)^{\lfloor(n-l)/2\rfloor}\cdot\prod\limits_{n_i\in\lambda}(-1)^{\nu_p(n_i)}}{(-1)^{\lfloor s_\lambda/2\rfloor}\cdot\prod\limits_{j=0}^m(-1)^{\lfloor(|\lambda_j|-l_j)/2\rfloor}\cdot\prod\limits_{n_i\mbox{ even}}\left(\frac{-1}{p}\right)^{\nu_p(n_i)}}.
\end{equation*}
Note that
$$
\sum_{n_i\in\lambda}\nu_p(n_i)=\sum_{j=0}^m\sum_{k=1}^{l_j}j=\sum\limits_{j=0}^mjl_j
$$
and
\begin{equation*}
\begin{split}
\prod_{n_i\mbox{ even}}\left(\frac{-1}{p}\right)^{\nu_p(n_i)}&=\prod_{j=0}^m\prod_{k=1}^{l_j}\left(\left(\frac{-1}{p}\right)^{n_{jk}-1}\right)^j=(-1)^{\frac{1}{2}(p-1)\sum\limits_{j=0}^m\sum\limits_{k=1}^{l_j}j(n_{jk}-1)}\\
&=(-1)^{\sum\limits_{j=0}^m\sum\limits_{k=1}^{l_j}\frac{1}{2}(p^j-1)(n_{jk}-1)}.
\end{split}
\end{equation*}
Since $s_\lambda$ is the number of non-self-associated bar partitions in $\{\lambda_j\mid0\leq j\leq m\}$, there are $s_\lambda$ odd integers in $\{|\lambda_j|-l_j\mid0\leq j\leq m\}$, so we have 
$$
\sum_{j=0}^m\lfloor(|\lambda_j|-l_j)/2\rfloor=\left(\sum_{j=0}^m\sum_{k=1}^{l_j}n_{jk}-\sum_{j=0}^ml_j-s_\lambda\right)/2=\left(\sum_{j=0}^m\sum_{k=1}^{l_j}n_{jk}-l-s_\lambda\right)/2.
$$
Moreover, since $s_\lambda\equiv n-l\mod2$, we have
$$
\lfloor(n-l)/2\rfloor-\lfloor s_\lambda/2\rfloor=(n-l-s_\lambda)/2=\left(\sum_{j=0}^m\sum_{k=1}^{l_j}n_{jk}p^j-l-s_\lambda\right)/2.
$$
Hence we conclude that
\begin{equation*}
\begin{split}
\frac{N_\lambda^\eta}{(-1)^{\lfloor s_\lambda/2\rfloor}N_{\lambda_0}^{\eta_0}N_{\lambda_1}^{\eta_1}N_{\lambda_2}^{\eta_2}\cdots N_{\lambda_m}^{\eta_m}}&=(-1)^{\sum\limits_{j=0}^m\frac{1}{2}(p^j-1)l_j}(-1)^{\sum\limits_{j=0}^mjl_j}=(-1)^{\frac{1}{2}(p-1)\sum\limits_{j=0}^mjl_j}(-1)^{\sum\limits_{j=0}^mjl_j}\\
&=(-1)^{\frac{1}{2}(p+1)\sum\limits_{j=0}^mjl_j}.
\end{split}
\end{equation*}
\end{proof}

\begin{remark}
\normalfont
If we wish to define another integer $M_\lambda^\eta$ such that 
$$
M_\lambda^\eta=(-1)^{\lfloor s_\lambda/2\rfloor}M_{\lambda_0}^{\eta_0}M_{\lambda_1}^{\eta_1}M_{\lambda_2}^{\eta_2}\cdots M_{\lambda_m}^{\eta_m}
$$
with $(-1)^{\frac{1}{2}(p+1)\sum\limits_{j=0}^mjl_j}$ omitted in Lemma~\ref{lemma:4.2}, we may define
$$
n''_j=\left\{\begin{array}{ll}
\left(\frac{-1}{p}\right)^{\nu_p(n_j)}(n_j)_{p'},&\mbox{ if $n_j$ is odd},\\
\left(\frac{-1}{p}\right)^{\nu_p(n_j)}(n_j)_{p'}/(2\eta),&\mbox{ if $n_j$ is even}
\end{array}\right.
$$
and define
$$
M_\lambda^\eta=(-1)^{\lfloor(n-l)/2\rfloor}\prod_{j=1}^ln''_j
$$
for a bar partition $\lambda=(n_1>n_2>\cdots>n_l>0)$. Since $\left(\frac{\left(\frac{-1}{p}\right)}{p}\right)=\left(\frac{-1}{p}\right)$, we have $\left(\frac{M_\lambda^\eta}{p}\right)=\left(\frac{N_\lambda^\eta}{p}\right)$ and under this definition
$$
M_\lambda^\eta=(-1)^{\lfloor s_\lambda/2\rfloor}M_{\lambda_0}^{\eta_0}M_{\lambda_1}^{\eta_1}M_{\lambda_2}^{\eta_2}\cdots M_{\lambda_m}^{\eta_m}.
$$
\end{remark}

\begin{theorem}\label{global}
The irreducible $p$-Brauer characters of $\tilde{S}_n^\eta$ in the spin block $B_\kappa$ can be labeled by bar partitions of $n$ in $\Lambda_\kappa$, i.e., we have a parametrization
$$
\mathrm{IBr}_p(B_\kappa)=\{\beta_\lambda^+,\beta_\lambda^-\mid\lambda\in\mathcal{P}_{n,\textrm{odd}}^{\textrm{str}}\cap\Lambda_\kappa\}\cup\{\beta_\lambda\mid\lambda\in\mathcal{P}_{n,\textrm{even}}^{\textrm{str}}\cap\Lambda_\kappa\}.
$$
for irreducible $p$-Brauer characters in $B_\kappa$ of $\tilde{S}_n^\eta$ such that
\begin{enumerate}
\item If $\lambda\in\mathcal{P}_{n,\textrm{odd}}^{\textrm{str}}\cap\Lambda_\kappa$, then $\beta_\lambda^-=\varepsilon\beta_\lambda^+\neq\beta_\lambda^+$ and $\sigma(\beta_\lambda^+)=\beta_\lambda^{\mu_\lambda}$, and
\item If $\lambda\in\mathcal{P}_{n,\textrm{even}}^{\textrm{str}}\cap\Lambda_\kappa$, then $\mathrm{Res}_{\tilde{A}_n}^{\tilde{S}_n^\eta}\beta_\lambda=\beta'^+_\lambda+\beta'^-_\lambda$ and $\sigma(\beta'^+_\lambda)=\beta'^{\mu_\lambda}_\lambda$.
\end{enumerate}
Here $\mu_\lambda=\left(\frac{N_\lambda^\eta}{p}\right)$ as before.
\end{theorem}

\begin{proof}
It suffices to prove the statement for positive defect spin blocks of $\tilde{S}_n^\eta$. Let $B'_\kappa$ be the block of $\tilde{A}_n$ covered by $B_\kappa$ (see \ref{void:blocks of double covers}). Let $\mathcal{T}$ be a set of conjugacy class representatives of $p$-elements in $\tilde{S}_n^\eta$ such that $1\in\mathcal{T}$, and let $\mathrm{IBr}_p(C_{\tilde{S}_n^\eta}(u),B_\kappa)$ be the set of irreducible $p$-Brauer characters of $C_{\tilde{S}_n^\eta}(u)$ belonging to a block whose Brauer correspondence in $\tilde{S}_n^\eta$ is $B_\kappa$. Let $\mathcal{T}'$ be a set of conjugacy class representatives of $p$-elements in $\tilde{A}_n$ such that $\mathcal{T}\subseteq\mathcal{T}'$, and let $\mathrm{IBr}_p(C_{\tilde{A}_n}(u),B'_\kappa)$ be the set of irreducible $p$-Brauer characters of $C_{\tilde{A}_n}(u)$ belonging to a block whose Brauer correspondence in $\tilde{S}_n^\eta$ is $B'_\kappa$. Set
\begin{equation*}
\begin{split}
\mathfrak{S}_1(\kappa)&=\{(u,\beta)\mid u\in\mathcal{T}-\{1\},\beta\in\mathrm{IBr}_p(C_{\tilde{S}_n^\eta}(u),B_\kappa)\},\\
\mathfrak{S}'_1(\kappa)&=\{(u',\beta')\mid u'\in\mathcal{T}'-\{1\},\beta'\in\mathrm{IBr}_p(C_{\tilde{A}_n}(u),B'_\kappa)\},\\
\mathfrak{S}(\kappa)&=\mathrm{IBr}_p(B_\kappa)\cup\mathfrak{S}_1(\kappa),\\
\mathfrak{S}'(\kappa)&=\mathrm{IBr}_p(B'_\kappa)\cup\mathfrak{S}'_1(\kappa),\\
\mathrm{Irr}_0(B_\kappa)&=\{\chi_\lambda\in\mathrm{Irr}(B_\kappa)\mid\lambda\in\Lambda_\kappa\},\\
\mathrm{Irr}_0(B'_\kappa)&=\{\chi'\in\mathrm{Irr}(B'_\kappa)\mid\mbox{$\chi'$ is covered by a character in }\mathrm{Irr}_0(B_\kappa)\},\\
\mathrm{Irr}_1(B_\kappa)&=\{\chi_\lambda\in\mathrm{Irr}(B_\kappa)\mid\lambda\notin\Lambda_\kappa\},\\
\mathrm{Irr}_1(B'_\kappa)&=\{\chi'\in\mathrm{Irr}(B'_\kappa)\mid\mbox{$\chi'$ is covered by a character in }\mathrm{Irr}_1(B_\kappa)\}.
\end{split}
\end{equation*}

We use induction on $n$. For any strict partition $\lambda$ of $n$, define $n_{jk}$ and $\lambda_j$ as in Lemma~\ref{lemma:4.2} and let $u_\lambda$ be the $p$-element in $\mathcal{T}$ whose image in $S_n$ consists of $|\lambda_j|$ copies of $p^j$-cycles. Let $s_\lambda$ be the number of non-self-associated strict partitions in $\{\lambda_j\mid0\leq j\leq m\}$. Then $s_\lambda\equiv n-l\mod2$. Applying the inductive hypothesis for each $\lambda_j$, we can define a bijection
\begin{eqnarray*}
\mathrm{Irr}_1(B_\kappa)&\to&\mathfrak{S}_1(\kappa),\\
\chi_\lambda&\mapsto&(u_\lambda,\beta_{\lambda_0}*\beta_{\lambda_1}*\cdots*\beta_{\lambda_m}).
\end{eqnarray*}
One checks that $\varepsilon\chi_\lambda=\chi_\lambda$ if and only if $\lambda$ is self-associated, or equivalently, $s_\lambda$ is even. This, in turn, is equivalent to $\varepsilon(\beta_{\lambda_0}*\beta_{\lambda_1}*\cdots*\beta_{\lambda_m})=\beta_{\lambda_0}*\beta_{\lambda_1}*\cdots*\beta_{\lambda_m}$.

Recall that for $\lambda=(n_1>n_2>\cdots>n_l>0)$, we have
$$
N_\lambda^\eta=(-1)^{\lfloor(n-l)/2\rfloor}\cdot\prod_{n_j\mbox{ odd}}(-1)^{\nu_p(n_j)}(n_j)_{p'}\cdot\prod_{n_j\mbox{ even}}\frac{(-1)^{\nu_p(n_j)}(n_j)_{p'}}{2\eta}.
$$
By Lemma \ref{lemma:4.2}, $N_\lambda^\eta=(-1)^{\frac{1}{2}(p+1)\sum\limits_{j=0}^mjl_j}\cdot(-1)^{\lfloor s_\lambda/2\rfloor}N_{\lambda_0}^{\eta_0}N_{\lambda_1}^{\eta_1}N_{\lambda_2}^{\eta_2}\cdots N_{\lambda_m}^{\eta_m}$.  Hence
$$
\left(\frac{N_\lambda^\eta}{p}\right)=\left(\frac{-1}{p}\right)^{\lfloor s_\lambda/2\rfloor}\left(\frac{N_{\lambda_0}^{\eta_0}}{p}\right)\left(\frac{N_{\lambda_1}^{\eta_1}}{p}\right)\left(\frac{N_{\lambda_2}^{\eta_2}}{p}\right)\cdots\left(\frac{N_{\lambda_m}^{\eta_m}}{p}\right).
$$
By Theorem~\ref{symbol_of_Humphreys_product}, we know that the action of $\sigma$ or $\varepsilon$ has the same number of fixed points on $\mathrm{Irr}_1(B_\kappa)$ and $\mathfrak{S}_1(\kappa)$, and the action of $\sigma$ or $\tilde{S}_n^\eta/\tilde{A}_n$ has the same fixed points on $\mathrm{Irr}_1(B'_\kappa)$ and $\mathfrak{S}'_1(\kappa)$.

By Brauer's Permutation Lemma (see \cite[Chapter~3, Lemma~2.18]{NT}) and Brauer's Second Main Theorem (see \cite[Chapter~5, Theorem~4.2]{NT}), the action of $\sigma$ or $\varepsilon$ has the same number of fixed points on $\mathrm{Irr}(B_\kappa)$ and $\mathfrak{S}(\kappa)$, and the action of $\sigma$ or $\tilde{S}_n^\eta/\tilde{A}_n$ has the same fixed points on $\mathrm{Irr}(B'_\kappa)$ and $\mathfrak{S}'(\kappa)$. Hence the action of $\sigma$ or $\varepsilon$ has the same number of fixed points on $\mathrm{Irr}_0(B_\kappa)$ and $\mathrm{IBr}(B_\kappa)$, and the action of $\sigma$ or $\tilde{S}_n^\eta/\tilde{A}_n$ has the same fixed points on $\mathrm{Irr}_0(B'_\kappa)$ and $\mathrm{IBr}(B'_\kappa)$.

The action of $\sigma$ on $\mathrm{Irr}_0(B_\kappa)$ and $\mathrm{Irr}_0(B'_\kappa)$ has been determined in Theorem~\ref{symbol_of_irr_+} and Theorem~\ref{symbol_of_irr_-}. We complete the proof.
\end{proof}

\section{Radical Subgroups of Double Covers of Symmetric Groups}\label{section5}

Let $\eta\in\{\pm1\}$. As before, we denote by $\tilde{S}_n^+$ and $\tilde{S}_n^-$ the double covers of $\tilde{S}_n$ such that
$$
\tilde{S}_n^\eta=\left\langle s_1,\ldots,s_{n-1}\mid s_i^2=\eta,(s_is_{i+1})^3=\eta,[s_i,s_j]=-1\mbox{ if }|i-j|\geq2\right\rangle.
$$

Let $p$ be an odd prime. The $p$-radical subgroups of $S_n$ and their normalizers have been determined in \cite{AF} and \cite{Fon}. The $p$-radical subgroups and weights of $\tilde{S}_n^-$ have been determined in \cite{MO91}. We restated the results and provide the construction of weights for both $\tilde{S}_n^+$ and $\tilde{S}_n^-$. To prove the blockwise Galois Alperin weight conjecture for $\tilde{S}_n^\pm$ and $\tilde{A}_n$, we will determine the action of Galois group on weights in next sections, using the Legendre symbol.

\begin{void}
{\rm	\textbf{Radical Subgroups of $S_n$.} 
Let $c$ be a positive integer. Denote by $R_c$ the elementary abelian $p$-subgroup of order $p^c$ in its left regular representation. For a sequence $\mathbf{c}=(c_1,c_2,\ldots,c_r)$ of positive integers, set
$$
R_{\mathbf{c}}=R_{c_1}\wr R_{c_2}\wr\cdots\wr R_{c_r}.
$$
Let $|\mathbf{c}|=c_1+c_2+\cdots+c_r$. Then $R_{\mathbf{c}}$ is a subgroup of $S_{p^{|\mathbf{c}|}}$. Moreover, let $N_{\mathbf{c}}=N_{S_{p^{|\mathbf{c}|}}}(R_{\mathbf{c}})$ be the normalizer of $R_{\mathbf{c}}$ in $S_{p^{|\mathbf{c}|}}$. Then
$$
N_{\mathbf{c}}/R_{\mathbf{c}}\cong\mathrm{GL}(c_1,p)\times\mathrm{GL}(c_2,p)\times\cdots\times\mathrm{GL}(c_r,p).
$$

The $p$-radical subgroups of $S_n$ have a decomposition
$$
R=I_{n_0}\times R_{\mathbf{c}_1}^{e_1}\times R_{\mathbf{c}_2}^{e_2}\times\cdots\times R_{\mathbf{c}_l}^{e_l}
$$
such that $n=n_0+e_1p^{|\mathbf{c}_1|}+e_2p^{|\mathbf{c}_2|}+\cdots+e_lp^{|\mathbf{c}_l|}$, where $I_{n_0}$ is the trivial subgroup of $S_{n_0}$. Denote by $C=C_{S_n}(R)$ the centralizer of $R$ in $S_n$ and $N=N_{S_n}(R)$ the normalizer of $R$ in $S_n$. Then
$$
RC=S_{n_0}\times R_{\mathbf{c}_1}^{e_1}\times R_{\mathbf{c}_2}^{e_2}\times\cdots\times R_{\mathbf{c}_l}^{e_l}
$$
and
$$
N=S_{n_0}\times N_{\mathbf{c}_1}\wr S_{e_1}\times N_{\mathbf{c}_2}\wr S_{e_2}\times\cdots\times N_{\mathbf{c}_l}\wr S_{e_l}.
$$
}
\end{void}

\begin{void}
{\rm	\textbf{Radical Subgroups of $\tilde{S}_n^\eta$.} 
Let $\tilde{S}_n^\eta\to S_n$ be the natural homomorphism. If $H$ is a subgroup of $S_n$, we denote by $\tilde{H}$ the preimage of $H$ in $\tilde{S}_n^\eta$. If $R$ is a $p$-subgroup of $S_n$, then we have $O_p(\tilde{R})\cong R$ and we view $R$ as a $p$-subgroup of $\tilde{S}_n^\eta$ via this isomorphism. By the structure of $p$-radical subgroups of $S_n$, we know that there is a decomposition
$$
R=I_{n_0}\times R_{\mathbf{c}_1}^{e_1}\times R_{\mathbf{c}_2}^{e_2}\times\cdots\times R_{\mathbf{c}_l}^{e_l}
$$
for $p$-radical subgroups of $\tilde{S}_n^\eta$. Let $C=C_{S_n}(R)$ and $N=N_{S_n}(R)$. Then we have $\tilde{C}=C_{\tilde{S}_n^\eta}(R)$ and $\tilde{N}=N_{\tilde{S}_n^\eta}(R)$. Moreover,
$$
R\tilde{C}=\tilde{S}_{n_0}^\eta\times R_{\mathbf{c}_1}^{e_1}\times R_{\mathbf{c}_2}^{e_2}\times\cdots\times R_{\mathbf{c}_l}^{e_l}
$$
and
$$
\tilde{N}=\tilde{S}_{n_0}^\eta*\widetilde{N_{\mathbf{c}_1}\wr S_{e_1}}*\widetilde{N_{\mathbf{c}_2}\wr S_{e_2}}*\cdots*\widetilde{N_{\mathbf{c}_l}\wr S_{e_l}},
$$
where $*$ is the twisted central product in Section~\ref{Humphreys_product}.
}
\end{void}

\section{Representations of Humphreys Product Components in Local Subgroups}\label{section6}
Recall that if $\mathbf{c}=(c_1,c_2,\ldots,c_r)$ is a finite sequence of positive integers, we have
$$
N_{\mathbf{c}}/R_{\mathbf{c}}=\mathrm{GL}(c_1,p)\times\mathrm{GL}(c_2,p)\times\cdots\times\mathrm{GL}(c_r,p).
$$
Following the argument in \cite[Lemma~2.3]{MO91}, we know that $\mathrm{GL}(c_j,p)=\mathrm{SL}(c_j,p)\rtimes\langle \bar{t}_j\rangle$ with $\bar{t}_j$ being an odd permutation in $S_n$ consisting of $p^{|\mathbf{c}|-1}$ cycles of length $p-1$. Moreover, let $t_j$ be a preimage of $\bar{t}_j$ in $\tilde{S}_n^\eta$. We have
$$
\widetilde{\mathrm{GL}(c_j,p)}=\mathrm{SL}(c_j,p)\rtimes\langle t_j,-1\rangle.
$$
Set $\tilde{T}=\langle t_1,t_2,\ldots,t_r,-1\rangle$ and $N'_{\mathbf{c}}$ be the subgroup of $\tilde{N}_{\mathbf{c}}$ such that
$$
N'_{\mathbf{c}}/R_{\mathbf{c}}=\mathrm{SL}(c_1,p)\times\mathrm{SL}(c_2,p)\times\cdots\times\mathrm{SL}(c_r,p).
$$
Then we have
$$
\tilde{N}_{\mathbf{c}}/R_{\mathbf{c}}=N'_{\mathbf{c}}/R_{\mathbf{c}}\rtimes\tilde{T}=(\mathrm{SL}(c_1,p)\times\mathrm{SL}(c_2,p)\times\cdots\times\mathrm{SL}(c_r,p))\rtimes\tilde{T}.
$$
Note that $N'_{\mathbf{c}}$ is isomorphic to its image in $S_n$ since $-1\notin N'_{\mathbf{c}}$.

\begin{lemma}\label{generator_of_T}
\begin{enumerate}
\item For any $1\leq j\leq r$, we have $t_j^{p-1}=\left(\frac{(-1)^{|\mathbf{c}|}2\eta}{p}\right)$;
\item For any $1\leq j,j'\leq r$, we have $[t_j,t_{j'}]=1$.
\end{enumerate}
In particular, we have
$$
\tilde{T}=\left\langle t_1,t_2,\ldots,t_r,-1\ \middle|\ t_j^{p-1}=\left(\frac{(-1)^{|\mathbf{c}|}2\eta}{p}\right),[t_j,t_{j'}]=1,1\leq j,j'\leq r\right\rangle.
$$
\end{lemma}

\begin{proof}
\begin{enumerate}
\item Since $\bar{t}_j$ is an odd permutation consisting of $p^{|\mathbf{c}|-1}$ cycles of length $p-1$, we know that $(\bar{t}_j)^{\frac{p-1}{2}}$ is a permutation consisting of $\frac{1}{2}(p-1)p^{|c|-1}$ two-cycles. Hence we have
\begin{equation*}
\begin{split}
t_j^{p-1}&=\left(t_j^{\frac{p-1}{2}}\right)^2=(-1)^{\frac{1}{2}\cdot\frac{1}{2}(p-1)p^{|\mathbf{c}|-1}\cdot\left(\frac{1}{2}(p-1)p^{|\mathbf{c}|-1}-1\right)}\cdot\eta^{\frac{1}{2}(p-1)p^{|\mathbf{c}|-1}}\\
&=(-1)^{\frac{1}{2}\cdot\frac{1}{2}(p-1)\cdot\left(\frac{1}{2}(p-1)-1\right)}\cdot(-1)^{\frac{1}{2}(p-1)\cdot\frac{1}{2}(p^{|\mathbf{c}|-1}-1)}\cdot\eta^{\frac{1}{2}(p-1)}\\
&=\left(\frac{-2}{p}\right)\cdot\left(\frac{-1}{p}\right)^{|\mathbf{c}|-1}\cdot\left(\frac{\eta}{p}\right)=\left(\frac{(-1)^{|\mathbf{c}|}2\eta}{p}\right).
\end{split}
\end{equation*}
\item Note that a cycle of length $p-1$ can be decomposed as a product of $p-2$ two-cycles. There exists decompositions
\begin{equation*}
\begin{split}
t_j&=t_{j1}t_{j2}\cdots t_{j,p-2},\\
t_{j'}&=t_{j'1}t_{j'2}\cdots t_{j',p-2}
\end{split}
\end{equation*}
such that for any $1\leq k,k'\leq p-2$, the images of $t_{jk}$ and $t_{j'k'}$ in $S_n$ are permutations consisting of $p^{|\mathbf{c}|-1}$ two-cycles, and they have $4(p^{|\mathbf{c}|-2})$ non-fixed points in common. Thus
$$
[t_{jk},t_{j'k'}]=(-1)^{p^{|\mathbf{c}|-2}}\cdot(-1)^{(p^{|\mathbf{c}|-2}(p-2))^2}=1.
$$
and we have $[t_j,t_{j'}]=\prod_{k,k'}[t_{jk},t_{j'k'}]=1$.
\end{enumerate}
\end{proof}

\begin{remark}
{\rm		If $\eta=-1$, the above lemma is exactly \cite[Lemma~2.3~(d)]{MO91}.}
\end{remark}

As before, let $\widetilde{N_{\mathbf{c}}\wr S_e}$ be the preimage of $N_{\mathbf{c}}\wr S_e$ in $\tilde{S}_n^\eta$. Since $N_{\mathbf{c}}\wr S_e=(N_{\mathbf{c}})^e\rtimes S_e$, we have
$$
\widetilde{N_{\mathbf{c}}\wr S_e}=\tilde{N}_{\mathbf{c}}^{*e}\cdot\tilde{S}_e^{\eta'}
$$
for some $\eta'\in\{\pm1\}$. Here $*e$ means the twisted central product of $e$ copies of $\tilde{N_{\mathbf{c}}}$.

\begin{lemma}
Let $s_1,s_2,\cdots,s_{e-1}$ be the generators of $\tilde{S}_e^{\eta'}$ whose images are the simple reflections in $S_e$. Then we have
\begin{equation*}
\begin{split}
s_k^2&=\left(\frac{(-1)^{|\mathbf{c}|}}{p}\right)\eta,\mbox{ for }1\leq k\leq e,\\
(s_ks_{k+1})^3&=\left(\frac{(-1)^{|\mathbf{c}|}}{p}\right)\eta,\mbox{ for }1\leq k\leq e-1,\\
[s_k,s_{k'}]&=-1,\mbox{ for }|k-k'|\geq2,1\leq k,k'\leq e.
\end{split}
\end{equation*}
In particular, we have $\eta'=\left(\frac{(-1)^{|\mathbf{c}|}}{p}\right)\eta$.
\end{lemma}

\begin{proof}
Note that the image of $s_k$ in $S_n$ consists of $p^{|\mathbf{c}|}$ two-cycles. We have
$$
s_k^2=(-1)^{\frac{1}{2}\cdot p^{|\mathbf{c}|}\cdot(p^{|\mathbf{c}|}-1)}\cdot\eta^{p^{|\mathbf{c}|}}=(-1)^{\frac{1}{2}(p^{|\mathbf{c}|}-1)}\eta=\left(\frac{(-1)^{|\mathbf{c}|}}{p}\right)\eta,
$$
and
$$
(s_ks_{k+1})^3=\left((-1)^{\frac{1}{2}\cdot p^{|\mathbf{c}|}\cdot(p^{|\mathbf{c}|}-1)}\right)^3\cdot\eta^{p^{|\mathbf{c}|}}=\left(\frac{(-1)^{|\mathbf{c}|}}{p}\right)\eta,
$$
and
$$
[s_k,s_{k'}]=(-1)^{p^{|\mathbf{c}|}\cdot p^{|\mathbf{c}|}}=-1.
$$
This completes the proof.
\end{proof}

Note that
$$
\left(\widetilde{N_{\mathbf{c}}\wr S_e}\right)/(N'_{\mathbf{c}})^e\cong(\underbrace{\tilde{T}*\tilde{T}*\cdots*\tilde{T}}_e)\cdot\tilde{S}_e^{\eta'}.
$$
By Lemma~\ref{generator_of_T}, we may choose $\{t_{11},t_{12},\ldots,t_{1r},-1\}$ to be a set of generaters of the first $\tilde{T}$ such that $t_{1j}^{p-1}=\left(\frac{(-1)^{|\mathbf{c}|}2\eta}{p}\right)$ and $[t_{1j},t_{1j'}]=1$, and we define $t_{k+1,j}=(-1)\cdot s_{k}^{-1}t_{kj}s_{k}$ for any $1\leq k\leq e-1$. Then we have
\begin{equation*}
\begin{split}
t_{kj}^{p-1}&=\left(\frac{(-1)^{|\mathbf{c}|}2\eta}{p}\right),\\
[t_{kj},t_{kj'}]&=1,\\
[t_{kj},t_{k'j'}]&=-1\mbox{ for }k'\neq k.
\end{split}
\end{equation*}
In other words, we have
$$
\tilde{T}^{*e}=\left\langle t_{kj},-1\ \middle|\ t_{kj}^{p-1}=\left(\frac{(-1)^{|\mathbf{c}|}2\eta}{p}\right),[t_{kj},t_{k'j'}]=\left\{\begin{array}{rl}
1,&\mbox{ if }k'=k\\
-1,&\mbox{ if }k'\neq k
\end{array}\right.\right\rangle.
$$
Note that $p\nmid|\tilde{T}^{*e}|=2(q-1)^e$. Moreover, the action of $\tilde{S}_e^{\eta'}$ on $\tilde{T}^{*e}$ is defined by
\begin{equation*}
\begin{split}
t_{k+1,j}&=(-1)\cdot s_k^{-1}t_{kj}s_k=(-1)\cdot s_kt_{kj}s_k^{-1},\\
[t_{kj},s_{k'}]&=-1\mbox{ for }k'\neq k-1,k.
\end{split}
\end{equation*}

\begin{lemma}\label{irr_of_Te}
\begin{enumerate}
\item There are $(p-1)^e$ linear characters of $\tilde{T}^{*e}$;
\item If $e$ is even, then $\tilde{T}^{*e}$ has $((p-1)/2)^e$ non-linear irreducible characters, all of which are of degree $2^{e/2}$, nontrivial on $\{\pm1\}$, and vanish outside the center;
\item If $e$ is odd, then $\tilde{T}^{*e}$ has $2((p-1)/2)^e$ non-linear irreducible characters, all of which are of degree $2^{(e-1)/2}$, nontrivial on $\{\pm1\}$, and vanish outside the center.
\end{enumerate}
\end{lemma}

\begin{proof}
\begin{enumerate}
\item This follows directly from the fact that $\tilde{T}^{*e}/\{\pm1\}$ is an abelian group of order $(p-1)^e$.
\item Let $e$ be an even number. Then $|Z(\tilde{T}^{*e})|=2((p-1)/2)^e$, so we have
$$
|\mathrm{Cl}(\tilde{T}^{*e})|=|Z(\tilde{T}^{*e})|+\frac{1}{2}(|\tilde{T}^{*e}|-|Z(\tilde{T}^{*e})|)=(p-1)^e+((p-1)/2)^e.
$$
Hence $\tilde{T}^{*e}$ has $((p-1)/2)^e$ irreducible characters nontrivial on $\{\pm1\}$. Let $d_1\geq d_2\geq\cdots\geq d_{((p-1)/2)^e}$ be the degrees of these irreducible characters. Then we have
$$
d_1^2\cdot|Z(\tilde{T}^{*e})|\leq|\tilde{T}^{*e}|.
$$
Thus $d_1\leq2^{e/2}$, and the equality holds if and only if the corresponding character vanishes outside the center. Moreover, we have
\begin{equation*}
\begin{split}
2(p-1)^e&=|\tilde{T}^{*e}|=(p-1)^e+\sum_{j=1}^{((p-1)/2)^e}d_j^2\\
&\leq(p-1)^e+\sum_{j=1}^{((p-1)/2)^e}d_1^2\\
&\leq(p-1)^e+((p-1)/2)^e\cdot(2^{e/2})^2=2(p-1)^e,
\end{split}
\end{equation*}
so $d_1=d_2=\cdots=d_{((p-1)/2)^e}=2^{e/2}$.
\item Let $e$ be an odd number. Then $|Z(\tilde{T}^{*e})|=4((p-1)/2)^e$, so we have
$$
|\mathrm{Cl}(\tilde{T}^{*e})|=|Z(\tilde{T}^{*e})|+\frac{1}{2}(|\tilde{T}^{*e}|-|Z(\tilde{T}^{*e})|)=(p-1)^e+2((p-1)/2)^e.
$$
Hence $\tilde{T}^{*e}$ has $2((p-1)/2)^e$ irreducible characters nontrivial on $\{\pm1\}$. Let $d_1\geq d_2\geq\cdots\geq d_{2((p-1)/2)^e}$ be the degrees of these irreducible characters. Then we have
$$
d_1^2\cdot|Z(\tilde{T}^{*e})|\leq|\tilde{T}^{*e}|.
$$
Thus $d_1\leq2^{(e-1)/2}$, and the equality holds if and only if the corresponding character vanishes outside the center. Moreover, we have
\begin{equation*}
\begin{split}
2(p-1)^e&=|\tilde{T}^{*e}|=(p-1)^e+\sum_{j=1}^{2((p-1)/2)^e}d_j^2\\
&\leq(p-1)^e+\sum_{j=1}^{2((p-1)/2)^e}d_1^2\\
&\leq(p-1)^e+2((p-1)/2)^e\cdot(2^{(e-1)/2})^2=2(p-1)^e,
\end{split}
\end{equation*}
so $d_1=d_2=\cdots=d_{2((p-1)/2)^e}=2^{(e-1)/2}$.
\end{enumerate}
\end{proof}

Let
$\sigma_x=\left(\begin{array}{cc}
0&1\\
1&0
\end{array}\right),\sigma_y=\left(\begin{array}{cr}
0&-\mathrm{i}\\
\mathrm{i}&0
\end{array}\right),\sigma_z=\left(\begin{array}{cr}
1&0\\
0&-1
\end{array}\right)$
be Pauli matrices. Set $e_0=\lfloor e/2\rfloor$ and
\begin{equation*}
\begin{split}
F_{2k_0-1}&=I_2^{\otimes(e_0-k_0)}\otimes\sigma_x\otimes\sigma_z^{\otimes(k_0-1)},\\
F_{2k_0}&=I_2^{\otimes(e_0-k_0)}\otimes\sigma_y\otimes\sigma_z^{\otimes(k_0-1)},\\
F_{2e_0+1}&=\sigma_z^{\otimes e_0}
\end{split}
\end{equation*}
for any $1\leq k_0\leq e_0$. Then $F_k^2=I_2$ and $[F_k,F_{k'}]=-I_2$ for any $1\leq k<k'\leq 2e_0+1$, and $F_1F_2\cdots F_{2e_0+1}=\mathrm{i}^{e_0}I_{2^{e_0}}$. Moreover, let
$$
E=\underbrace{\sigma_y\otimes\sigma_x\otimes\sigma_y\otimes\sigma_x\otimes\cdots}_{e_0}.
$$
We can check that $E^2=I_2$ and $EF_kE=(-1)^{e_0}(-1)^{k-1}F_k$.

%{\color{red}	Recall that} $\varepsilon\colon\tilde{S}_n\to\{\pm1\}$ denotes the sign character of $\tilde{S}_n$.

\begin{lemma}\label{rep_of_Te}
Assume that $e\geq2$. Let $\alpha_1,\alpha_2,\ldots,\alpha_e\in\mathrm{Irr}(\tilde{T})$ be $e$ irreducible characters of $\tilde{T}$ such that $\alpha_k(-1)=-1$ for all $1\leq k\leq e$. Then the group homomorphism
\begin{eqnarray*}
\Delta_{\alpha_1,\alpha_2,\ldots,\alpha_e}\colon\underbrace{\tilde{T}*\cdots*\tilde{T}}_e&\to&\mathrm{GL}_{2^{e_0}}(\bar{\mathbb{Q}}_p),\\
t_{kj}&\mapsto&\alpha_k(t_j)F_k
\end{eqnarray*}
is an irreducible representation of $\tilde{T}^{*e}$. Let $\psi_{\alpha_1,\alpha_2,\ldots,\alpha_e}$ be the irreducible character of $\Delta_{\alpha_1,\alpha_2,\ldots,\alpha_e}$.
\begin{enumerate}
\item If $e$ is even, then $\psi_{\alpha_1,\alpha_2,\ldots,\alpha_e}=\varepsilon\psi_{\alpha_1,\alpha_2,\ldots,\alpha_e}$ is the only one irreducible character of $\tilde{T}^{*e}$ whose restriction to the $k$-th $\tilde{T}$ is $2^{e_0-1}(\alpha_k+\varepsilon\alpha_k)$;
\item If $e$ is odd, then $\psi_{\alpha_1,\alpha_2,\ldots,\alpha_e}\neq\varepsilon\psi_{\alpha_1,\alpha_2,\ldots,\alpha_e}$ are the only two irreducible characters of $\tilde{T}^{*e}$ whose restrictions to the $k$-th $\tilde{T}$ are $2^{e_0-1}(\alpha_k+\varepsilon\alpha_k)$.
\end{enumerate}
In particular, all non-linear irreducible representations of $\tilde{T}^{*e}$ can be constructed in this way. For other irreducible characters $\alpha'_1,\alpha'_2,\ldots,\alpha'_e\in\mathrm{Irr}(\tilde{T})$ such that $\alpha'_k(-1)=-1$, we have $\psi_{\alpha'_1,\alpha'_2,\ldots,\alpha'_e}=\psi_{\alpha_1,\alpha_2,\ldots,\alpha_e}$ or $\varepsilon\psi_{\alpha_1,\alpha_2,\ldots,\alpha_e}$ if and only if $\alpha'_k=\alpha_k$ or $\varepsilon\alpha_k$ for all $1\leq k\leq e$.
\end{lemma}

\begin{proof}
We can check that for every $1\leq j,j'\leq r$ and $1\leq k<k'\leq e$, we have
\begin{equation*}
\begin{split}
\Delta_{\alpha_1,\alpha_2,\ldots,\alpha_e}(t_{kj})^{p-1}&=\alpha_k(t_j^{p-1})I_{2^{e_0}},\\
[\Delta_{\alpha_1,\alpha_2,\ldots,\alpha_e}(t_{kj}),\Delta_{\alpha_1,\alpha_2,\ldots,\alpha_e}(t_{kj'})]&=I_{2^{e_0}},\\
[\Delta_{\alpha_1,\alpha_2,\ldots,\alpha_e}(t_{kj}),\Delta_{\alpha_1,\alpha_2,\ldots,\alpha_e}(t_{k'j'})]&=-I_{2^{e_0}}.
\end{split}
\end{equation*}
Hence $\Delta_{\alpha_1,\alpha_2,\ldots,\alpha_e}$ is a well-defined group representation. Moreover, we have
$$
\Delta_{\alpha_1,\alpha_2,\ldots,\alpha_e}(-1)=-I_{2^{e_0}}.
$$
Since all irreducible characters of $\tilde{T}^{*e}$ which are nontrivial on $\{\pm1\}$ are of degree $2^{e_0}$ by Lemma~\ref{irr_of_Te}, we know that $\Delta_{\alpha_1,\alpha_2,\ldots,\alpha_e}$ is an irreducible representation of $\tilde{T}^{*e}$. Moreover, since the eigenvalues of $F_k$ are $2^{e_0-1}$ copies of $1$ and $2^{e_0-1}$ copies of $-1$, the restriction of $\psi_{\alpha_1,\alpha_2,\ldots,\alpha_e}$ to the $k$-th $\tilde{T}$ is $2^{e_0-1}(\alpha_k+\varepsilon\alpha_k)$. In particular, we know that $\psi_{\alpha'_1,\alpha'_2,\ldots,\alpha'_e}\notin\{\psi_{\alpha_1,\alpha_2,\ldots,\alpha_e},\varepsilon\psi_{\alpha_1,\alpha_2,\ldots,\alpha_e}\}$ if $\alpha'_{k_0}\notin\{\alpha_{k_0},\varepsilon\alpha_{k_0}\}$ for some $1\leq k_0\leq e$, by comparing their restrictions to the $k_0$-th $\tilde{T}$.
\begin{enumerate}
\item If $e$ is even, then we have constructed all $((p-1)/2)^e$ irreducible characters which are nontrivial on $\{\pm1\}$. In particular, we have
$$
\psi_{\alpha_1,\ldots,\alpha_k,\ldots,\alpha_e}=\psi_{\alpha_1,\ldots,\varepsilon\alpha_k,\ldots,\alpha_e}=\varepsilon\psi_{\alpha_1,\ldots,\alpha_k,\ldots,\alpha_e}.
$$
\item If $e$ is odd, then $t_{1j_1}t_{2j_2}\cdots t_{ej_e}\in Z(\tilde{T}^{*e})$ is a central element for any $1\leq j_1,j_2,\ldots,j_e\leq r$. Moreover, we have
\begin{equation*}
\begin{split}
\Delta_{\alpha_1,\alpha_2,\ldots,\alpha_e}(t_{1j_1}t_{2j_2}\cdots t_{ej_e})&=\alpha_1(t_{1j_1})F_1\cdot\alpha_2(t_{2j_2})F_2\cdot\cdots\cdot\alpha_e(t_{ej_e})F_e\\
&=\mathrm{i}^{e_0}\left(\prod_{k=1}^e\alpha_k(t_{j_k})\right)I_{2^{e_0}}.
\end{split}
\end{equation*}
Thus we have $\psi_{\alpha_1,\ldots,\alpha_e}(t_{1j_1}t_{2j_2}\cdots t_{ej_e})=(2\mathrm{i})^{e_0}\prod\limits_{k=1}^e\alpha_k(t_{kj})\neq0$ and
\begin{equation*}
\begin{split}
&\psi_{\alpha_1,\ldots,\varepsilon\alpha_k,\ldots,\alpha_e}(t_{1j_1}t_{2j_2}\cdots t_{ej_e})=\varepsilon\psi_{\alpha_1,\ldots,\alpha_k,\ldots,\alpha_e}(t_{1j_1}t_{2j_2}\cdots t_{ej_e})\\
=&-\psi_{\alpha_1,\ldots,\alpha_k,\ldots,\alpha_e}(t_{1j_1}t_{2j_2}\cdots t_{ej_e})=-(2\mathrm{i})^{e_0}\left(\prod_{k=1}^e\alpha_k(t_{j_k})\right),
\end{split}
\end{equation*}
so we have constructed all $2((p-1)/2)^e$ irreducible characters nontrivial on $\{\pm1\}$. In particular, we have
$$
\psi_{\alpha_1,\ldots,\varepsilon\alpha_k,\ldots,\alpha_e}=\varepsilon\psi_{\alpha_1,\ldots,\alpha_k,\ldots,\alpha_e}\neq\psi_{\alpha_1,\ldots,\alpha_k,\ldots,\alpha_e}
$$
and both of them have the same restrictions on each $\tilde{T}$.
\end{enumerate}
\end{proof}

\begin{proposition}
Let $\alpha\in\mathrm{Irr}(\tilde{T})$ be an irreducible character of $\tilde{T}$ such that $\alpha(-1)=-1$. Then the irreducible representation $\Delta_{\alpha,\alpha,\ldots,\alpha}$ of $\tilde{T}^{*e}$ can be extended to an irreducible representation
\begin{eqnarray*}
\Delta_\alpha^+\colon\tilde{T}^{*e}\cdot\tilde{S}_e^{\eta'}&\to&\mathrm{GL}_{2^{e_0}}(\bar{\mathbb{Q}}_p),\\
t_{kj}&\mapsto&\alpha(t_j)F_k,\\
s_k&\mapsto&\frac{1}{\sqrt{2\eta'}}(F_k-F_{k+1})
\end{eqnarray*}
of $\left(\widetilde{N_{\mathbf{c}}\wr S_e}\right)/(N'_{\mathbf{c}})^e\cong\tilde{T}^{*e}\cdot\tilde{S}_e^{\eta'}$.
\end{proposition}

\begin{proof}
We can check that
\begin{equation*}
\begin{split}
\Delta_\alpha^+(t_{kj})&=\Delta_{\alpha,\alpha,\ldots,\alpha}(t_{kj}),\\
\Delta_\alpha^+(s_k)^2&=\eta'I_{2^{e_0}},\\
(\Delta_\alpha^+(s_k)\Delta_\alpha^+(s_{k+1}))^3&=\eta'I_{2^{e_0}},\\
[\Delta_\alpha^+(s_k),\Delta_\alpha^+(s_{k'})]&=-I_{2^{e_0}}\mbox{ if }|k-k'|\geq2,\\
\Delta_\alpha^+(t_{k+1,j})&=-\Delta_\alpha^+(s_k)^{-1}\Delta_\alpha^+(t_{kj})\Delta_\alpha^+(s_k),\\
[\Delta_\alpha^+(t_{kj}),\Delta_\alpha^+(s_{k'})]&=-I_{2^{e_0}}\mbox{ if }k'\neq k-1,k.
\end{split}
\end{equation*}
This completes the proof.
\end{proof}

\begin{notation}
{\rm For the rest of the paper we let $t$ be a positive integer which is not divisible by $p$ such that for every finite group $G$ below, we have $4|G|_{p'}\mid t$. Let $\omega$ be a primitive $t$-th root of unity in $\bar{\mathbb{Q}}_p$. Let $\sigma$ be the field automorphism of $\Q(\omega)$ such that $\sigma(\omega)=\omega^p$. 
}
\end{notation}

Recall that $\eta'=\left(\frac{(-1)^{|\mathbf{c}|}}{p}\right)\eta$ and $t_j^{p-1}=\left(\frac{(-1)^{|\mathbf{c}|}2\eta}{p}\right)$. We have
\begin{equation*}
\begin{split}
\sigma(\alpha(t_j))&=\left(\frac{(-1)^{|\mathbf{c}|}2\eta}{p}\right)\alpha(t_j),\\
\sigma(\sqrt{2\eta'})&=\left(\frac{(-1)^{|\mathbf{c}|}2\eta}{p}\right)\sqrt{2\eta'},\\
\sigma(F_{2k_0-1})&=F_{2k_0-1},\\
\sigma(F_{2k_0})&=\left(\frac{-1}{p}\right)F_{2k_0}.
\end{split}
\end{equation*}
Define another representation $\Delta_\alpha^-\colon\tilde{T}^{*e}\cdot\tilde{S}_e^{\eta'}\to\mathrm{GL}_{2^{e_0}}(\bar{\mathbb{Q}}_p)$ by $\Delta_\alpha^-=\varepsilon\Delta_\alpha^+$.

\begin{proposition}
The irreducible representation $\Delta_\alpha^+$ is equivalent to $\Delta_\alpha^-$ if and only if $e$ is even.
\end{proposition}

\begin{proof}
\begin{enumerate}
\item If $e=2e_0$ is even, we can check that
\begin{equation*}
\begin{split}
\Delta_\alpha^-(t_{kj})&=F_{2e_0+1}\Delta_\alpha^+(t_{kj})F_{2e_0+1},\\
\Delta_\alpha^-(s_k)&=F_{2e_0+1}\Delta_\alpha^+(s_k)F_{2e_0+1},
\end{split}
\end{equation*}
so we have $\Delta_\alpha^-=F_{2e_0+1}\Delta_\alpha^+F_{2e_0+1}$.
\item If $e=2e_0+1$ is odd, we can check that
\begin{equation*}
\begin{split}
\Delta_\alpha^+(t_{1j}t_{2j}\cdots t_{ej})&=\mathrm{i}^{e_0}\alpha(t_j)^eI_{2^{e_0}},\\
\Delta_\alpha^-(t_{1j}t_{2j}\cdots t_{ej})&=-\mathrm{i}^{e_0}\alpha(t_j)^eI_{2^{e_0}},
\end{split}
\end{equation*}
which are not conjugate.
\end{enumerate}
\end{proof}

\begin{proposition}\label{symbol_for_odd}
The representation $\sigma(\Delta_\alpha^+)$ is equivalent to $\Delta_\alpha^{\mu'}$, where
$$
{\mu'}=\left(\frac{(-1)^{|\mathbf{c}|+e_0}2\eta}{p}\right).
$$
In particular, $\sigma(\Delta_\alpha^+)$ is equivalent to $\Delta_\alpha^+$ if $e$ is even.
\end{proposition}

\begin{proof}
\begin{enumerate}
\item If $p\equiv1\mod4$, then we have $\sigma(F_k)=F_k$ and $\left(\frac{-1}{p}\right)=1$. Hence
\begin{equation*}
\begin{split}
\sigma(\Delta_\alpha^+(t_{kj}))&=\left(\frac{(-1)^{|\mathbf{c}|}2\eta}{p}\right)\Delta_\alpha^+(t_{kj}),\\
\sigma(\Delta_\alpha^+(s_k))&=\left(\frac{(-1)^{|\mathbf{c}|}2\eta}{p}\right)\Delta_\alpha^+(s_k),
\end{split}
\end{equation*}
so we have $\sigma(\Delta_\alpha^+)=\Delta_\alpha^{\mu'}$ with
$$
{\mu'}=\left(\frac{(-1)^{|\mathbf{c}|}2\eta}{p}\right)=\left(\frac{(-1)^{|\mathbf{c}|+e_0}2\eta}{p}\right);
$$
\item If $p\equiv3\mod4$, then we have $\sigma(F_k)=(-1)^{k-1}F_k$ and $\left(\frac{-1}{p}\right)=-1$. Recall that $EF_kE=(-1)^{e_0}(-1)^{k-1}F_k$ with
$$
E=\underbrace{\sigma_y\otimes\sigma_x\otimes\sigma_y\otimes\sigma_x\otimes\cdots}_{e_0}.
$$
Hence
\begin{equation*}
\begin{split}
\sigma(\Delta_\alpha^+(t_{kj}))&=\left(\frac{(-1)^{|\mathbf{c}|}2\eta}{p}\right)(-1)^{e_0}\cdot E\Delta_\alpha^+(t_{kj})E,\\
\sigma(\Delta_\alpha^+(s_k))&=\left(\frac{(-1)^{|\mathbf{c}|}2\eta}{p}\right)(-1)^{e_0}\cdot E\Delta_\alpha^+(s_k)E,
\end{split}
\end{equation*}
so we have $\sigma(\Delta_\alpha^+)=E\Delta_\alpha^{\mu'} E$ with
$$
{\mu'}=\left(\frac{(-1)^{|\mathbf{c}|}2\eta}{p}\right)(-1)^{e_0}=\left(\frac{(-1)^{|\mathbf{c}|+e_0}2\eta}{p}\right).
$$
\end{enumerate}
\end{proof}

\begin{proposition}\label{symbol_for_even}
Assume $e$ is even. Suppose that
$$
\mathrm{Res}_{(\tilde{T}^{*e}\cdot\tilde{S}_e^{\eta'})\cap\tilde{A}_n}^{\tilde{T}^{*e}\cdot\tilde{S}_e^{\eta'}}\Delta_\alpha^+=\Delta'^+_\alpha\oplus\Delta'^-_\alpha.
$$
Then the two irreducible representations $\Delta'^\pm_\alpha$ of $(\tilde{T}^{*e}\cdot\tilde{S}_e^{\eta'})\cap\tilde{A}_n$ can be chosen such that
\begin{equation*}
\begin{split}
\Delta'^+_\alpha(t_{1j}t_{2j}\cdots t_{ej})&=\mathrm{i}^{e_0}\alpha(t_j)^{e}I_{2^{e_0-1}},\\
\Delta'^-_\alpha(t_{1j}t_{2j}\cdots t_{ej})&=-\mathrm{i}^{e_0}\alpha(t_j)^{e}I_{2^{e_0-1}}.
\end{split}
\end{equation*}
In particular, $\sigma(\Delta'^+_\alpha)$ is equivalent to $\Delta'^{\mu''}_\alpha$ with
$$
\mu''=\left(\frac{(-1)^{e_0}}{p}\right).
$$
\end{proposition}

\begin{proof}
Note that $t_{1j}t_{2j}\cdots t_{ej}\in Z((\tilde{T}^{*e}\cdot\tilde{S}_e^{\eta'})\cap\tilde{A}_n)$, so $\Delta'^\pm_\alpha(t_{1j}t_{2j}\cdots t_{ej})$ are both scalar matrices. Since the eigenvalues of
$$
\Delta^+_\alpha(t_{1j}t_{2j}\cdots t_{ej})=\alpha(t_j)^eF_1F_2\cdots F_{2e_0}=\mathrm{i}^{e_0}\alpha(t_j)^e\cdot\sigma_z^{\otimes e_0}
$$
are $2^{e_0-1}$ copies of $\mathrm{i}^{e_0}\alpha(t_j)^e$ and $-\mathrm{i}^{e_0}\alpha(t_j)^e$, this completes the proof.
\end{proof}

Recall that
$$
(N_{\mathbf{c}}/R_{\mathbf{c}})\wr S_e=(\mathrm{GL}(c_1,p)\times\mathrm{GL}(c_2,p)\times\cdots\times\mathrm{GL}(c_r,p))\wr S_e.
$$
The product of Steinberg characters of $\mathrm{GL}(c_j,p)$ can be extended to an irreducible character of $(N_{\mathbf{c}}/R_{\mathbf{c}})\wr S_e$, denoted by $\mathrm{St}_{\mathbf{c}}$. Then $\mathrm{St}_{\mathbf{c}}$ can be viewed as an irreducible character of $\left(\widetilde{N_{\mathbf{c}}\wr S_e}\right)/R_{\mathbf{c}}^e$ trivial on $\{\pm1\}$ such that $\mathrm{Res}_{(N'_{\mathbf{c}}/R_{\mathbf{c}})^e}^{(\widetilde{N_{\mathbf{c}}\wr S_e})/R_{\mathbf{c}}^e}\mathrm{St}_{\mathbf{c}}$ is the only one $p$-defect-zero irreducible character of $(N'_{\mathbf{c}}/R_{\mathbf{c}})^e$. Note that
$$
N'_{\mathbf{c}}/R_{\mathbf{c}}=\mathrm{SL}(c_1,p)\times\mathrm{SL}(c_2,p)\times\cdots\times\mathrm{SL}(c_r,p).
$$
Moreover, all values of $\mathrm{St}_{\mathbf{c}}\in\mathrm{Irr}((\widetilde{N_{\mathbf{c}}\wr S_e})/R_{\mathbf{c}}^e)$ are integers.

Let $\alpha\in\mathrm{Irr}(\tilde{T})$ be an irreducible character of $\tilde{T}$ such that $\alpha(-1)=-1$ and let $\psi_\alpha^+$, $\psi_\alpha^-$ be the irreducible character of $\Delta_\alpha^+$, $\Delta_\alpha^-$, respectively (note that $\psi_\alpha^+=\psi_\alpha^-$ if and only if $e$ is even). Then $\psi_\alpha^+$ can be viewed as an irreducible character of $\left(\widetilde{N_{\mathbf{c}}\wr S_e}\right)/R_{\mathbf{c}}^e$ by the inflation via
$$
1\to(N'_{\mathbf{c}}/R_{\mathbf{c}})^e\to\left(\widetilde{N_{\mathbf{c}}\wr S_e}\right)/R_{\mathbf{c}}^e\to\tilde{T}^{*e}\cdot\tilde{S}_e^{\eta'} \to1.
$$
Hence $\mathrm{St}_{\mathbf{c}}\psi_\alpha^+$ is an irreducible character of $\left(\widetilde{N_{\mathbf{c}}\wr S_e}\right)/R_{\mathbf{c}}^e$ such that its restriction to $\tilde{N}_{\mathbf{c}}^{*e}/R_{\mathbf{c}}^e$ is a $p$-defect-zero irreducible character of $\tilde{N}_{\mathbf{c}}^{*e}/R_{\mathbf{c}}^e$. Moreover, let $\kappa$ be a $p$-bar-core partition of $e$ and let $\chi_\kappa$ be the irreducible character of $S_e$ corresponding to $\kappa$. Then $\chi_\kappa$ can be viewed as an irreducible character of $\left(\widetilde{N_{\mathbf{c}}\wr S_e}\right)/R_{\mathbf{c}}^e$ trivial on $\tilde{N}_{\mathbf{c}}^{*e}/R_{\mathbf{c}}^e$ by inflation. Thus $\mathrm{St}_{\mathbf{c}}\psi_\alpha^+\chi_\kappa$ is a $p$-defect-zero character of $\left(\widetilde{N_{\mathbf{c}}\wr S_e}\right)/R_{\mathbf{c}}^e$. Moreover, all values of $\chi_\kappa$ are integers.

Set $\tilde{T}_0=\tilde{T}\cap\tilde{A}_n$ and $\tilde{M}_{\mathbf{c}}=N'_{\mathbf{c}}\rtimes\tilde{T}_0$, where $\tilde{A}_n$ is the kernel of $\varepsilon\colon\tilde{S}_n^\eta\to\{\pm1\}$, in other words, the double cover of the alternating group $A_n$. Then we have $\tilde{M}_{\mathbf{c}}=\tilde{N}_{\mathbf{c}}\cap\tilde{A}_n$ and
$$
\tilde{M}_{\mathbf{c}}^e=\underbrace{\tilde{M}_{\mathbf{c}}\cdot_{\langle-1\rangle}\tilde{M}_{\mathbf{c}}\cdot_{\langle-1\rangle}\cdots\ \cdot_{\langle-1\rangle}\tilde{M}_{\mathbf{c}}}_e\unlhd\tilde{N}_{\mathbf{c}}^{*e}=\underbrace{\tilde{N}_{\mathbf{c}}*\tilde{N}_{\mathbf{c}}*\cdots*\tilde{N}_{\mathbf{c}}}_e
$$
is a normal subgroup of $\tilde{N}_{\mathbf{c}}^{*e}$ of index $2^e$. Here $\cdot_{\langle-1\rangle}$ is the central product.

\begin{theorem}\label{irr_of_NcSe/Rc}
Let $\alpha\in\mathrm{Irr}(\tilde{T})$ be an irreducible character of $\tilde{T}$ such that $\alpha(-1)=-1$. View $\alpha$ as a linear character of $\tilde{N}_{\mathbf{c}}/R_{\mathbf{c}}$. If $\varphi\in\mathrm{Irr}((\widetilde{N_{\mathbf{c}}\wr S_e})/R_{\mathbf{c}}^e)$ is a $p$-defect-zero irreducible character such that $\varphi(-1)=-\varphi(1)$ and
$$
\mathrm{Res}_{(\tilde{M}_{\mathbf{c}}/R_{\mathbf{c}})^e}^{(\widetilde{N_{\mathbf{c}}\wr S_e})/R_{\mathbf{c}}^e}\varphi=\frac{\varphi(1)}{\mathrm{St}_{\mathbf{c}}(1)}\cdot\mathrm{Res}_{(\tilde{M}_{\mathbf{c}}/R_{\mathbf{c}})^e}^{(\widetilde{N_{\mathbf{c}}\wr S_e})/R_{\mathbf{c}}^e}\mathrm{St}_{\mathbf{c}}\cdot\left(\mathrm{Res}_{\tilde{M}_{\mathbf{c}}/R_{\mathbf{c}}}^{\tilde{N}_{\mathbf{c}}/R_{\mathbf{c}}}\alpha\right)^e,
$$
there is a $p$-core partition $\kappa$ of $e$ such that $\varphi=\mathrm{St}_{\mathbf{c}}\psi_\alpha^+\chi_\kappa$ or $\varphi=\mathrm{St}_{\mathbf{c}}\psi_\alpha^-\chi_\kappa$. Set
$$
\mu=\left(\frac{-1}{p}\right)^{e_0}\left(\frac{(-1)^{|\mathbf{c}|}2\eta}{p}\right)^e=\left\{\begin{array}{ll}
\mu',&\mbox{if }e=2e_0+1\mbox{ is odd},\\
\mu'',&\mbox{if }e=2e_0\mbox{ is even}.
\end{array}\right.
$$
\begin{enumerate}
\item If $e$ is odd, then $\varepsilon\varphi\neq\varphi$ and $\sigma(\mathrm{St}_{\mathbf{c}}\psi_\alpha^+\chi_\kappa)=\mathrm{St}_{\mathbf{c}}\psi_\alpha^\mu\chi_\kappa$;
\item If $e$ is even, then $\varepsilon\varphi=\varphi$ and
$$
\mathrm{Res}_{(\widetilde{N_{\mathbf{c}}\wr S_e}\cap\tilde{A}_n)/R_{\mathbf{c}}^e}^{(\widetilde{N_{\mathbf{c}}\wr S_e})/R_{\mathbf{c}}^e}\varphi=\varphi^++\varphi^-
$$
for two irreducible characters $\varphi^\pm$ such that $\sigma(\varphi^+)=\varphi^\mu$.
\end{enumerate}
\end{theorem}

\begin{proof}
Since $\varphi$ is a $p$-defect-zero irreducible character, we know that
$$
\mathrm{Res}_{\tilde{N}_{\mathbf{c}}^{*e}/R_{\mathbf{c}}^e}^{(\widetilde{N_{\mathbf{c}}\wr S_e})/R_{\mathbf{c}}^e}\varphi=\varphi_1+\varphi_2+\cdots+\varphi_m
$$
is a sum of $p$-defect zero irreducible characters of $(\tilde{N}_{\mathbf{c}}/R_{\mathbf{c}})^e$ conjugate to each other under the action of $\tilde{S}_e^{\eta'}$ such that $\varphi_j(-1)=-\varphi_j(1)$. Since $\mathrm{Res}_{(N'_{\mathbf{c}}/R_{\mathbf{c}})^e}^{\tilde{N}_{\mathbf{c}}^{*e}/R_{\mathbf{c}}^e}\mathrm{St}_{\mathbf{c}}$ is the only one $p$-defect-zero irreducible character of $(N'_{\mathbf{c}}/R_{\mathbf{c}})^e$, which can be extended to $\tilde{N}_{\mathbf{c}}^{*e}/R_{\mathbf{c}}^e$. We know that
$$
\varphi_j=\mathrm{Res}_{\tilde{N}_{\mathbf{c}}^{*e}/R_{\mathbf{c}}^e}^{(\widetilde{N_{\mathbf{c}}\wr S_e})/R_{\mathbf{c}}^e}\mathrm{St}_{\mathbf{c}}\cdot\psi_{\alpha_1,\alpha_2,\ldots,\alpha_e}
$$
for some $\alpha_1,\alpha_2,\ldots,\alpha_e\in\mathrm{Irr}(\tilde{T})$. Here $\psi_{\alpha_1,\alpha_2,\ldots,\alpha_e}$ is the irreducible character of $\tilde{T}^{*e}$ defined as in Lemma~\ref{rep_of_Te} and again we view it as an irreducible character of $\tilde{N}_{\mathbf{c}}^{*e}/R_{\mathbf{c}}^e$ by inflation. Moreover, since
$$
\mathrm{Res}_{(\tilde{M}_{\mathbf{c}}/R_{\mathbf{c}})^e}^{\tilde{N}_{\mathbf{c}}^{*e}/R_{\mathbf{c}}^e}\psi_{\alpha_1,\alpha_2,\ldots,\alpha_e}=2^{e_0}\cdot\mathrm{Res}_{\tilde{M}_{\mathbf{c}}/R_{\mathbf{c}}}^{\tilde{N}_{\mathbf{c}}/R_{\mathbf{c}}}\alpha_1\cdot\mathrm{Res}_{\tilde{M}_{\mathbf{c}}/R_{\mathbf{c}}}^{\tilde{N}_{\mathbf{c}}/R_{\mathbf{c}}}\alpha_2\cdot\ \cdots\ \cdot\mathrm{Res}_{\tilde{M}_{\mathbf{c}}/R_{\mathbf{c}}}^{\tilde{N}_{\mathbf{c}}/R_{\mathbf{c}}}\alpha_e.
$$
We know that $\mathrm{Res}_{\tilde{M}_{\mathbf{c}}/R_{\mathbf{c}}}^{\tilde{N}_{\mathbf{c}}/R_{\mathbf{c}}}\alpha_j=\mathrm{Res}_{\tilde{M}_{\mathbf{c}}/R_{\mathbf{c}}}^{\tilde{N}_{\mathbf{c}}/R_{\mathbf{c}}}\alpha$ and hence $\alpha_j=\alpha$ or $\alpha_j=\varepsilon\alpha$, so
$$
\varphi_j=\mathrm{Res}_{\tilde{N}_{\mathbf{c}}^{*e}/R_{\mathbf{c}}^e}^{(\widetilde{N_{\mathbf{c}}\wr S_e})/R_{\mathbf{c}}^e}\mathrm{St}_{\mathbf{c}}\cdot\psi_{\alpha,\alpha,\ldots,\alpha}\mbox{ or }\varphi_j=\mathrm{Res}_{\tilde{N}_{\mathbf{c}}^{*e}/R_{\mathbf{c}}^e}^{(\widetilde{N_{\mathbf{c}}\wr S_e})/R_{\mathbf{c}}^e}\mathrm{St}_{\mathbf{c}}\cdot\varepsilon\psi_{\alpha,\alpha,\ldots,\alpha},
$$
in other words,
$$
\varphi_j=\mathrm{Res}_{\tilde{N}_{\mathbf{c}}^{*e}/R_{\mathbf{c}}^e}^{(\widetilde{N_{\mathbf{c}}\wr S_e})/R_{\mathbf{c}}^e}\left(\mathrm{St}_{\mathbf{c}}\psi_\alpha^+\right)\mbox{ or }\varphi_j=\mathrm{Res}_{\tilde{N}_{\mathbf{c}}^{*e}/R_{\mathbf{c}}^e}^{(\widetilde{N_{\mathbf{c}}\wr S_e})/R_{\mathbf{c}}^e}\left(\mathrm{St}_{\mathbf{c}}\psi_\alpha^-\right).
$$
In particular, each $\varphi_j$ is invariant under the action of $\tilde{S}_e^{\eta'}$. Hence we have $\varphi_1=\varphi_2=\cdots=\varphi_m$. By comparing the degrees we have $m=\frac{\varphi(1)}{2^{e_0}\mathrm{St}_{\mathbf{c}}(1)}$ and
$$
\mathrm{Res}_{\tilde{N}_{\mathbf{c}}^{*e}/R_{\mathbf{c}}^e}^{(\widetilde{N_{\mathbf{c}}\wr S_e})/R_{\mathbf{c}}^e}\varphi=\frac{\varphi(1)}{2^{e_0}\mathrm{St}_{\mathbf{c}}(1)}\cdot\mathrm{Res}_{\tilde{N}_{\mathbf{c}}^{*e}/R_{\mathbf{c}}^e}^{(\widetilde{N_{\mathbf{c}}\wr S_e})/R_{\mathbf{c}}^e}\left(\mathrm{St}_{\mathbf{c}}\psi_\alpha^\pm\right).
$$
Since $\varphi$ is a $p$-defect-zero irreducible character of $(\widetilde{N_{\mathbf{c}}\wr S_e})/R_{\mathbf{c}}^e$, we know that $\varphi=\mathrm{St}_{\mathbf{c}}\psi_\alpha^\pm\chi_\kappa$ for some $p$-bar-core partition $\kappa$ of $e$.

\begin{enumerate}
\item If $e=2e_0+1$ is odd, by Proposition~\ref{symbol_for_odd} we have $\sigma(\mathrm{St}_{\mathbf{c}}\psi_\alpha^+\chi_\kappa)=\mathrm{St}_{\mathbf{c}}\psi_\alpha^\mu\chi_\kappa$ with
$$
\mu=\mu'=\left(\frac{(-1)^{|\mathbf{c}|+e_0}2\eta}{p}\right);
$$
\item If $e=2e_0$ is even and assume
$$
\mathrm{Res}_{(\tilde{T}^{*e}\cdot\tilde{S}_e^{\eta'})\cap\tilde{A}_n}^{\tilde{T}^{*e}\cdot\tilde{S}_e^{\eta'}}\left(\psi_\alpha^+\chi_\kappa\right)=\varphi'^++\varphi'^-,
$$
by Proposition~\ref{symbol_for_even} we have
$$
\varphi'^+(t_{1j}t_{2j}\cdots t_{ej})=-\varphi'^-(t_{1j}t_{2j}\cdots t_{ej})=2^{e_0-1}\mathrm{i}^{e_0}\alpha(t_j)^{e}\chi_\kappa(1),
$$
so $\sigma\left(\mathrm{Res}_{(\widetilde{N_{\mathbf{c}}\wr S_e}\cap\tilde{A}_n)/R_{\mathbf{c}}^e}^{(\widetilde{N_{\mathbf{c}}\wr S_e})/R_{\mathbf{c}}^e}\mathrm{St}_{\mathbf{c}}\cdot\varphi'^+\right)=\mathrm{Res}_{(\widetilde{N_{\mathbf{c}}\wr S_e}\cap\tilde{A}_n)/R_{\mathbf{c}}^e}^{(\widetilde{N_{\mathbf{c}}\wr S_e})/R_{\mathbf{c}}^e}\mathrm{St}_{\mathbf{c}}\cdot\varphi'^\mu$ with
$$
\mu=\mu''=\left(\frac{(-1)^{e_0}}{p}\right).
$$
Note that $\varphi^\pm=\mathrm{Res}_{(\widetilde{N_{\mathbf{c}}\wr S_e}\cap\tilde{A}_n)/R_{\mathbf{c}}^e}^{(\widetilde{N_{\mathbf{c}}\wr S_e})/R_{\mathbf{c}}^e}\mathrm{St}_{\mathbf{c}}\cdot\varphi'^\pm$. This completes the proof.
\end{enumerate}
\end{proof}

Let $\Gamma\subseteq\mathrm{Irr}(\tilde{T})$ be a set of representatives of irreducible characters of $\tilde{T}$ which are nontrivial on $\{\pm1\}$ under the action of multiplying $\varepsilon$. Then $\varepsilon\Gamma\cap\Gamma=\emptyset$. Denote by $\mathrm{dz}(\tilde{M}_{\mathbf{c}}/R_{\mathbf{c}}\mid-1)$ the set of $p$-defect-zero irreducible characters of $\tilde{M}_{\mathbf{c}}/R_{\mathbf{c}}$ which are nontrivial on $\{\pm1\}$. We have a bijection
$$
\Gamma\to\mathrm{dz}(\tilde{M}_{\mathbf{c}}/R_{\mathbf{c}}\mid-1)
$$
given by multiplying the Steinberg character of $N_{\mathbf{c}}/R_{\mathbf{c}}$ and restricting to $\tilde{M}_{\mathbf{c}}/R_{\mathbf{c}}$.

\section{Weights for the Covering Subgroups of Symmetric Groups}\label{section7}

Let
$$
R=I_{n_0}\times R_{\mathbf{c}_1}^{e_1}\times R_{\mathbf{c}_2}^{e_2}\times\cdots\times R_{\mathbf{c}_l}^{e_l}
$$
be a $p$-radical subgroup of $\tilde{S}_n$ such that $n=n_0+e_1p^{|\mathbf{c}_1|}+e_2p^{|\mathbf{c}_2|}+\cdots+e_lp^{|\mathbf{c}_l|}$, and let
$$
\tilde{N}=N_{\tilde{S}_n^\eta}(R)=\tilde{S}_{n_0}^\eta*\widetilde{N_{\mathbf{c}_1}\wr S_{e_1}}*\widetilde{N_{\mathbf{c}_2}\wr S_{e_2}}*\cdots*\widetilde{N_{\mathbf{c}_l}\wr S_{e_l}}
$$
be its normalizer in $\tilde{S}_n^{\eta}$. As before, set $\tilde{M}_{\mathbf{c}}=\tilde{N}_{\mathbf{c}}\cap\tilde{A}_n$ for any $\mathbf{c}\in\{\mathbf{c}_1,\mathbf{c}_2,\ldots,\mathbf{c}_l\}$. Then we have $|\tilde{N}_{\mathbf{c}}/\tilde{M}_{\mathbf{c}}|=2$ and
$$
\tilde{M}=\tilde{S}_{n_0}^\eta\cdot_{\langle-1\rangle}\tilde{M}_{\mathbf{c}_1}^{e_1}\cdot_{\langle-1\rangle}\tilde{M}_{\mathbf{c}_2}^{e_2}\cdot_{\langle-1\rangle}\cdots\ \cdot_{\langle-1\rangle}\tilde{M}_{\mathbf{c}_l}^{e_l}
$$
is a normal subgroup of $\tilde{N}$.

Fix a positive integer $d>0$. Define
$$
\mathcal{C}_d=\bigcup_{|\mathbf{c}|=d}\mathrm{dz}(\tilde{M}_{\mathbf{c}}/R_{\mathbf{c}}\mid-1),
$$
where $\mathrm{dz}(\tilde{M}_{\mathbf{c}}/R_{\mathbf{c}}\mid-1)$ is the set of $p$-defect-zero irreducible characters of $\tilde{M}_{\mathbf{c}}/R_{\mathbf{c}}$ which are nontrivial on $\{\pm1\}$, and the disjoint union runs over all finite sequences $\mathbf{c}$ of positive integers such that $|\mathbf{c}|=d$. Then we have
\begin{equation*}
\begin{split}
|\mathcal{C}_d|&=\frac{p-1}{2}+\binom{d-1}{1}\frac{(p-1)^2}{2}+\binom{d-1}{2}\frac{(p-1)^3}{2}+\cdots+\binom{d-1}{d-1}\frac{(p-1)^d}{2}\\
&=\frac{1}{2}(p-1)(p-1+1)^{d-1}=\frac{1}{2}(p-1)p^{d-1}.
\end{split}
\end{equation*}

Let $\Phi\in\mathrm{Irr}(\tilde{N}/R)$ be a $p$-defect-zero irreducible character of $\tilde{N}/R$ nontrivial on $\{\pm1\}$ and let $\phi\in\mathrm{Irr}(\tilde{M}/R)$ be an irreducible component of $\mathrm{Res}_{\tilde{M}/R}^{\tilde{N}/R}\Phi$. Then $\phi$ has a decomposition $\phi=\chi_\kappa\phi_1$ with $\chi_\kappa\in\mathrm{Irr}(\tilde{S}_{n_0}^\eta)$ and
$$
\phi_1=\prod_{d\geq1}\prod_{\psi\in\mathcal{C}_d}\psi^{e_\psi}\in\mathrm{Irr}(\tilde{M}_{\mathbf{c}_1}^{e_1}\cdot_{\langle-1\rangle}\tilde{M}_{\mathbf{c}_2}^{e_2}\cdot_{\langle-1\rangle}\cdots\ \cdot_{\langle-1\rangle}\tilde{M}_{\mathbf{c}_l}^{e_l})
$$
such that $\kappa$ is a $p$-bar-core bar partition (or strict partition) of $n_0$ and
$$
\sum_{\psi\in\mathrm{dz}(\tilde{M}_{\mathbf{c}_j}/R_{\mathbf{c}_j}\mid-1)}e_\psi=e_j\mbox{ for any }1\leq j\leq l.
$$
Recall that $\tilde{C}R/R=\tilde{S}_{n_0}^\eta$, where $\tilde{C}=C_{\tilde{S}_n^\eta}(R)$ is the centralizer of $R$. Let $B_\kappa$ be the spin block of $\tilde{S}_n^\eta$ corresponding to $\kappa$. Since $\mathrm{bl}(\chi_\kappa)^G=B_\kappa$, we know that $\mathrm{bl}(\Phi)^G=B_\kappa$. In other words, we have $(R,\Phi)$ is a $B_\kappa$-weight. For each $\psi\in\mathcal{C}_d$, we denote by $\mathbf{c}_\psi$ the finite sequence of integers in $\{\mathbf{c}_1,\mathbf{c}_2,\ldots,\mathbf{c}_l\}$ such that $\psi\in\mathrm{dz}(\tilde{M}_{\mathbf{c}_\psi}/R_{\mathbf{c}_\psi}\mid-1)$.

The action of $\tilde{N}/\tilde{M}$ on $\phi$ is a permutation of $\psi\in\mathcal{C}_d$ for each $d\geq1$. Denote by $\tilde{N}(\phi)$ the inertial group of $\phi$ in $\tilde{N}$. Then we have
$$
\tilde{N}(\phi)\leq\tilde{N}(\phi_1)=\tilde{S}_{n_0}^\eta*\prod_{d\geq1}^*\prod_{\psi\in\mathcal{C}_d}^*\widetilde{N_{\mathbf{c}_\psi}\wr S_{e_\psi}},
$$
where $\prod\limits^*$ is the twisted central product. Hence a $p$-defect-zero irreducible character of $\tilde{N}(\phi)/R$ is a Humphreys product of $p$-defect-zero irreducible characters of $\tilde{S}_{n_0}^\eta$ and all $\left(\widetilde{N_{\mathbf{c}_\psi}\wr S_{e_\psi}}\right)/R_{\mathbf{c}_\psi}^{e_\psi}$'s, and by Clifford Theorem the induction provides a bijection between $p$-defect-zero irreducible characters of $\tilde{N}/R$ over $\phi$ and that of $\tilde{N}(\phi)/R$ over $\phi$. By Theorem~\ref{irr_of_NcSe/Rc}, for each $\psi\in\mathcal{C}_d$, we have
\begin{enumerate}
\item If $e_\psi$ is odd, then $p$-defect-zero irreducible characters of $\left(\widetilde{N_{\mathbf{c}_\psi}\wr S_{e_\psi}}\right)/R_{\mathbf{c}_\psi}^{e_\psi}$ nontrivial on $\{\pm1\}$ over $\psi^{e_\psi}$ can be labeled as
$$
\mathrm{dz}((\widetilde{N_{\mathbf{c}_\psi}\wr S_{e_\psi}})/R_{\mathbf{c}_\psi}^{e_\psi}\mid-1,\psi^{e_\psi})=\left\{\Psi_{\psi,\kappa_\psi}^+,\Psi_{\psi,\kappa_\psi}^-\mid\kappa_\psi\mbox{ is a }p\mbox{-bar-core of }e_\psi\right\}
$$
such that $\Psi_{\psi,\kappa_\psi}^-=\varepsilon\Psi_{\psi,\kappa_\psi}^+\neq\Psi_{\psi,\kappa_\psi}^+$ and $\sigma(\Psi_{\psi,\kappa_\psi}^+)=\Psi_{\psi,\kappa_\psi}^{\mu_\psi}$ with
$$
\mu_\psi=\left(\frac{-1}{p}\right)^{\lfloor e_\psi/2\rfloor}\left(\frac{(-1)^{|\mathbf{c}_\psi|}2\eta}{p}\right)^{e_\psi}=\left(\frac{-1}{p}\right)^{(e_\psi-1)/2}\left(\frac{(-1)^{|\mathbf{c}_\psi|}2\eta}{p}\right);
$$
\item If $e_\psi$ is even, then $p$-defect-zero irreducible characters of $\left(\widetilde{N_{\mathbf{c}_\psi}\wr S_{e_\psi}}\right)/R_{\mathbf{c}_\psi}^{e_\psi}$ nontrivial on $\{\pm1\}$ over $\psi^{e_\psi}$ can be labeled as
$$
\mathrm{dz}((\widetilde{N_{\mathbf{c}_\psi}\wr S_{e_\psi}})/R_{\mathbf{c}_\psi}^{e_\psi}\mid-1,\psi^{e_\psi})=\left\{\Psi_{\psi,\kappa_\psi}\mid\kappa_\psi\mbox{ is a }p\mbox{-bar-core of }e_\psi\right\}
$$
such that $\mathrm{Res}_{(\widetilde{N_{\mathbf{c}}\wr S_e}\cap\tilde{A}_n)/R_{\mathbf{c}}^e}^{(\widetilde{N_{\mathbf{c}}\wr S_e})/R_{\mathbf{c}}^e}\Psi_{\psi,\kappa_\psi}=\Psi'^+_{\psi,\kappa_\psi}+\Psi'^-_{\psi,\kappa_\psi}$ and $\sigma(\Psi'^+_{\psi,\kappa_\psi})=\Psi'^{\mu_\psi}_{\psi,\kappa_\psi}$ with
$$
\mu_\psi=\left(\frac{-1}{p}\right)^{\lfloor e_\psi/2\rfloor}\left(\frac{(-1)^{|\mathbf{c}_\psi|}2\eta}{p}\right)^{e_\psi}=\left(\frac{-1}{p}\right)^{e_\psi/2}.
$$
\end{enumerate}

Define a map $f$ by
\begin{eqnarray*}
f\colon\bigcup_{d\geq1}\mathcal{C}_d&\to&\{p\mbox{-cores}\},\\
\psi&\mapsto&\kappa_\psi.
\end{eqnarray*}
Then we have
$$
n=|\kappa|+\sum_{d\geq1}\sum_{\psi\in\mathcal{C}_d}|f(\psi)|p^d,
$$
and the pair $(\kappa,f)$ defines one or two $p$-defect-zero irreducible characters of $\tilde{N}(\phi_1)/R$ over $\phi$ by the Humphreys product of $\chi_\kappa$ and $\Psi_{\psi,f(\psi)}^{\pm}$, $\Psi_{\psi,f(\psi)}$, according to the number of non-self-associated terms. More specifically, note that
$$
n-|\kappa|\equiv\sum_{d\geq1}\sum_{\psi\in\mathcal{C}_d}|f(\psi)|\equiv\#\left\{\psi\in\bigcup_{d\geq1}\mathcal{C}_d\ \middle|\ |f(\psi)|\mbox{ is odd}\right\}\mod2.
$$
We define
\begin{equation*}
\begin{split}
\mathrm{sym}_0&=\left\{\begin{array}{rl}
1,&\mbox{if }\chi_\kappa\mbox{ is self-associated},\\
-1,&\mbox{if }\chi_\kappa\mbox{ is non-self-associated},
\end{array}\right.\\
\mathrm{sym}&=\mathrm{sym}_0\cdot(-1)^{n-|\kappa|}.
\end{split}
\end{equation*}
Then $(\kappa,f)$ defines one $p$-defect-zero irreducible character if $\mathrm{sym}=1$ and two $p$-defect-zero irreducible characters if $\mathrm{sym}=-1$.

\begin{theorem}\label{symbol_of_weight}
Fix a pair $(\kappa,f)$ and let $n=|\kappa|+pw$. Set $\mathcal{C}_d^+=\{\psi\in\mathcal{C}_d\mid|f(\psi)|\mbox{ is even}\}$ and $\mathcal{C}_d^-=\{\psi\in\mathcal{C}_d\mid|f(\psi)|\mbox{ is odd}\}$. Denote
$$
\tilde{N}_1(\phi_1)=\prod_{d\geq1}^*\prod_{\psi\in\mathcal{C}_d}^*\widetilde{N_{\mathbf{c}_\psi}\wr S_{e_\psi}}
$$
and
$$
\mu_w=\left(\frac{-1}{p}\right)^{\lceil w/2\rceil}\left(\frac{2\eta}{p}\right)^w.
$$
Let
$$
\Psi_f=\prod_{d\geq1}^*\prod_{\psi\in\mathcal{C}_d^+}^*\Psi_{\psi,f(\psi)}*\prod_{d\geq1}^*\prod_{\psi\in\mathcal{C}_d^-}^*\Psi_{\psi,f(\psi)}^+
$$
be the Humphreys product of $\Psi_{\psi,f(\psi)}$ or $\Psi_{\psi,f(\psi)}^+$ according to $\psi\in\mathcal{C}_d^+$ or $\mathcal{C}_d^-$.
\begin{enumerate}
\item If $w$ is odd, then $\Psi_f$ is a non-self-associated $p$-defect-zero irreducible character of $\tilde{N}_1(\phi_1)$. Let $\Psi_f^+=\Psi_f$ and $\Psi_f^-=\varepsilon\Psi_f$. Then $\sigma(\Psi_f^+)=\Psi_f^{\mu_w}$.
\item If $w$ is even, then $\Psi_f$ is a self-associated $p$-defect-zero irreducible character of $\tilde{N}_1(\phi)$. Let $\mathrm{Res}_{\tilde{N}_1(\phi_1)\cap\tilde{A}_n}^{\tilde{N}_1(\phi_1)}\Psi_f=\Psi'^+_f+\Psi'^-_f$. Then $\sigma(\Psi'^+_f)=\Psi'^{\mu_w}_f$.
\end{enumerate}
\end{theorem}

\begin{proof}
Write $n_0=|\kappa|$ and $n_1=n-n_0=pw$. Note that
$$
\lfloor pw/2\rfloor-\lfloor w/2\rfloor=(p-1)w/2.
$$
We have
$$
\left(\frac{-1}{p}\right)^{\lfloor pw/2\rfloor}=\left(\frac{-1}{p}\right)^{\lfloor w/2\rfloor}\left(\frac{-1}{p}\right)^{(p-1)w/2}=\left(\frac{-1}{p}\right)^{\lfloor w/2\rfloor+w}=\left(\frac{-1}{p}\right)^{\lceil w/2\rceil}.
$$
Hence
$$
\mu_w=\left(\frac{-1}{p}\right)^{\lceil w/2\rceil}\left(\frac{2\eta}{p}\right)^w=\left(\frac{-1}{p}\right)^{\lfloor n_1/2\rfloor}\left(\frac{2\eta}{p}\right)^{n_1}.
$$

Now we can use Theorem~\ref{symbol_of_Humphreys_product}. Since $\sigma(\Psi_{\psi,\kappa_\psi}^+)=\Psi_{\psi,\kappa_\psi}^{\mu_\psi}$ for $\psi\in\mathcal{C}_d^-$ and $\sigma(\Psi'^+_{\psi,\kappa_\psi})=\Psi'^{\mu_\psi}_{\psi,\kappa_\psi}$ for $\psi\in\mathcal{C}_d^+$, it suffices to show that
$$
\mu_w=\left(\frac{-1}{p}\right)^{\lfloor s/2\rfloor}\prod_{d\geq1}\prod_{\psi\in\mathcal{C}_d}\mu_\psi,
$$
where $s=\sum_{d\geq1}|\mathcal{C}_d^-|$.

Denote by $e'_\psi=\lfloor|f(\psi)|/2\rfloor$. Then we have $|f(\psi)|=2e'_\psi+1$ if $\psi\in\mathcal{C}_d^-$ and  $|f(\psi)|=2e'_\psi$ if $\psi\in\mathcal{C}_d^+$, so
$$
n_1=\sum_{d\geq1}\sum_{\psi\in\mathcal{C}_d^+}2e'_\psi p^d+\sum_{d\geq1}\sum_{\psi\in\mathcal{C}_d^-}(2e'_\psi+1)p^d.
$$
Hence we have $n_1\equiv s\mod2$ and
$$
\frac{n_1}{2}=\sum_{d\geq1}\sum_{\psi\in\mathcal{C}_d}e'_\psi p^d+\sum_{d\geq1}\sum_{\psi\in\mathcal{C}_d^-}\frac{p^d}{2}.
$$
By direct calculation, we have
\begin{equation*}
\begin{split}
\prod_{d\geq1}\prod_{\psi\in\mathcal{C}_d}\mu_\psi=&\prod_{d\geq1}\prod_{\psi\in\mathcal{C}_d^+}\left(\frac{-1}{p}\right)^{e'_\psi}\cdot\prod_{d\geq1}\prod_{\psi\in\mathcal{C}_d^-}\left(\frac{-1}{p}\right)^{e'_\psi}\left(\frac{(-1)^{d}2\eta}{p}\right)\\
=&\prod_{d\geq1}\prod_{\psi\in\mathcal{C}_d}\left(\frac{-1}{p}\right)^{e'_\psi p^d}\cdot\prod_{d\geq1}\prod_{\psi\in\mathcal{C}_d^-}\left(\frac{-1}{p}\right)^d\cdot\left(\frac{2\eta}{p}\right)^s\\
=&\prod_{d\geq1}\prod_{\psi\in\mathcal{C}_d}\left(\frac{-1}{p}\right)^{e'_\psi p^d}\cdot\prod_{d\geq1}\prod_{\psi\in\mathcal{C}_d^-}\left(\frac{-1}{p}\right)^{(p^d-1)/2}\cdot\left(\frac{2\eta}{p}\right)^{n_1}\\
=&\left(\frac{-1}{p}\right)^{(n_1-s)/2}\left(\frac{2\eta}{p}\right)^{n_1}.
\end{split}
\end{equation*}
So we conclude that
$$
\mu_w=\left(\frac{-1}{p}\right)^{\lfloor n_1/2\rfloor}\left(\frac{2\eta}{p}\right)^{n_1}=\left(\frac{-1}{p}\right)^{\lfloor s/2\rfloor}\prod_{d\geq1}\prod_{\psi\in\mathcal{C}_d}\mu_\psi.
$$
\end{proof}

\begin{remark}
{\rm	We point out that the value of $\mu_w$ only depends on $w$.}
\end{remark}

\begin{corollary}\label{local}
Let $n=|\kappa|+pw$ and
$$
\Phi_{\kappa,f}=\mathrm{Ind}_{\tilde{N}(\phi_1)/R}^{\tilde{N}/R}(\chi_\kappa*\Psi_f)
$$
be the $p$-defect-zero irreducible character of $\tilde{N}/R$ defined by $(\kappa,f)$.
\begin{enumerate}
\item If $\kappa\in\mathcal{P}_{n_0,\textrm{even}}^{\textrm{str}}$ and $w$ is odd, or $\kappa\in\mathcal{P}_{n_0,\textrm{odd}}^{\textrm{str}}$ is non-self-associated and $w$ is even, then $\Phi_{\kappa,f}$ is non-self-associated. Moreover, let $\Phi_{\kappa,f}^+=\Phi_{\kappa,f}$ and $\Phi_{\kappa,f}^-=\varepsilon\Phi_{\kappa,f}$. Then $\sigma(\Phi_{\kappa,f}^+)=\Phi_{\kappa,f}^{\mu_{\kappa,f}}$ with
$$
\mu_{\kappa,f}=\left(\frac{N_\kappa^\eta}{p}\right)\cdot\left(\frac{-1}{p}\right)^{\lceil w/2\rceil}\left(\frac{2\eta}{p}\right)^w.
$$
\item If $\kappa\in\mathcal{P}_{n_0,\textrm{even}}^{\textrm{str}}$ and $w$ is even, then $\Phi_{\kappa,f}$ is self-associated. Moreover, let $\mathrm{Res}_{(\tilde{N}\cap\tilde{A}_n)/R}^{\tilde{N}/R}\Phi_{\kappa,f}=\Phi'^+_{\kappa,f}+\Phi'^-_{\kappa,f}$. Then $\sigma(\Phi'^+_{\kappa,f})=\Phi'^{\mu_{\kappa,f}}_{\kappa,f}$ with
$$
\mu_{\kappa,f}=\left(\frac{N_\kappa^\eta}{p}\right)\cdot\left(\frac{-1}{p}\right)^{w/2}\left(\frac{2\eta}{p}\right)^w.
$$
\item If $\kappa\in\mathcal{P}_{n_0,\textrm{odd}}^{\textrm{str}}$ and $w$ is odd, then $\Phi_{\kappa,f}$ is self-associated. Moreover, let $\mathrm{Res}_{(\tilde{N}\cap\tilde{A}_n)/R}^{\tilde{N}/R}\Phi_{\kappa,f}=\Phi'^+_{\kappa,f}+\Phi'^-_{\kappa,f}$. Then $\sigma(\Phi'^+_{\kappa,f})=\Phi'^{\mu_{\kappa,f}}_{\kappa,f}$ with
$$
\mu_{\kappa,f}=\left(\frac{N_\kappa^\eta}{p}\right)\cdot\left(\frac{-1}{p}\right)^{\lfloor w/2\rfloor}\left(\frac{2\eta}{p}\right)^w.
$$
\end{enumerate}
\end{corollary}

\begin{proof}
By Mackey's formula, we have
$$
\mathrm{Res}_{(\tilde{N}\cap\tilde{A}_n)/R}^{\tilde{N}/R} \mathrm{Ind}_{\tilde{N}(\phi_1)/R}^{\tilde{N}/R}(\chi_\kappa*\Psi_f)=
\mathrm{Ind}_{(\tilde{N}(\phi_1)\cap\tilde{A}_n)/R}^{(\tilde{N}\cap\tilde{A}_n)/R}
\mathrm{Res}_{(\tilde{N}(\phi_1)\cap\tilde{A}_n)/R}^{\tilde{N}(\phi_1)/R}(\chi_\kappa*\Psi_f).
$$
Note that $\phi_1$ is invariant under the action of $\sigma$. The results follow directly from Theorem~\ref{symbol_of_weight}, Theorem~\ref{global}, Theorem~\ref{symbol_of_Humphreys_product}, and the Clifford Theorem.
\end{proof}

\section{Proof of Theorem~\ref{main}}\label{section8}

In this section, we complete the proof of Theorem~\ref{main}. It suffices to consider a positive defect block $B_\kappa$ of $\tilde{S}_n^\eta$ and the block $B'_\kappa$ of $\tilde{A}_n$ covered by $B_\kappa$. Recall that
$$
\Lambda_\kappa=\left\{\lambda=(n_1>n_2>\cdots>n_l>0)\ \middle|\ \substack{\mbox{ $\kappa$ is the $p$-bar-core of $\lambda$},\\ \mbox{$p\nmid n_j$ for all $1\leq j\leq l$}  }\right\}.
$$

Fix an order of $p$-defect-zero irreducible characters in $\mathcal{C}_d$. A map
$$
f\colon\bigcup_{d\geq1}\mathcal{C}_d\to\{p\mbox{-cores}\}
$$
defines a $p$-quotient with empty runner $0$. Hence there is a bijection
\begin{eqnarray*}
&\left\{\lambda\in\Lambda_\kappa\mid\lambda\mbox{ is a bar partition of }n\right\}&\\
&\updownarrow&\\
&\left\{f\colon\bigcup_{d\geq1}\mathcal{C}_d\to\{p\mbox{-cores}\}\ \middle|\ n=|\kappa|+\sum_{d\geq1}\sum_{\psi\in\mathcal{C}_d}|f(\psi)|p^d\right\}&
\end{eqnarray*}
using \cite[Proposition~4.2]{Ols93}.

By Theorem~\ref{symbol_of_IBr}, Theorem~\ref{global} and Corollary~\ref{local}, we have
$$
\mu_{\kappa,f}=\mu_\lambda=\left(\frac{N_\lambda^\eta}{p}\right)=\left\{\begin{array}{ll}
\left(\frac{N_\kappa^\eta}{p}\right)\cdot\left(\frac{-1}{p}\right)^{\lceil w/2\rceil}\left(\frac{2\eta}{p}\right)^w,&\kappa\in\mathcal{P}_{n_0,\textrm{even}}^{\textrm{str}},\\
\left(\frac{N_\kappa^\eta}{p}\right)\cdot\left(\frac{-1}{p}\right)^{\lfloor w/2\rfloor}\left(\frac{2\eta}{p}\right)^w,&\kappa\in\mathcal{P}_{n_0,\textrm{odd}}^{\textrm{str}}
\end{array}\right.
$$
with $n=|\kappa|+pw$, and the following results hold.
\begin{enumerate}
\item There is a bijection between irreducible $p$-Brauer characters in $B_\kappa$ and conjugacy classes of weights in $B_\kappa$, which is invariant under the action of $\varepsilon$ and $\sigma$.
\item There is a bijection between irreducible $p$-Brauer characters in $B'_\kappa$ and conjugacy classes of weights in $B'_\kappa$, which is invariant under the action of $\tilde{S}_n^\eta/\tilde{A}_n$ and $\sigma$.
\end{enumerate}
Hence we have proved Theorem~\ref{main} for $n\neq6$. The case $n=6$ can be checked individually.

\section*{Acknowledgement}

We wish to thank Pham Huu Tiep for helpful and valuable discussions. The first and the third authors are supported by National Key R \& D Program of China Grant (No. 2020YFE0204200) and NSFC Grants (No. 12431001 and No. 12350710787). The second author is supported by Fundamental Research Funds for the Central Universities (CCNU24XJ028) and China Postdoctoral Science Foundation (2024M761095).


\begin{thebibliography}{10}

\bibitem{AF} J. Alperin, P. Fong, Weights for symmetric and general linear groups, \emph{J. Algebra} \textbf{131} (1990), 2--22.

\bibitem{BG} O. Brunat, J. B. Gramain, Basic sets for the double covering groups of the symmetric and alternating groups in odd characteristic, \emph{Algebra Represent. Theory} \textbf{23} (2020), 193--207.

\bibitem{BN} O. Brunat, R. Nath, The Navarro conjecture for alternating groups, \emph{Algebra \& Number Theory} \textbf{15} (2021), 821--862.

\bibitem{DH} Y. Du, X. Huang, The blockwise Galois Alperin Weight Conjecture for symmetric and alternating groups, \emph{Math. Z.} \textbf{311} (2025), Paper No. 3.

\bibitem{Fon} P. Fong, The Alperin Weight Conjecture for symmetric and general linear groups revisited, \emph{J. Algebra} \textbf{558} (2020), 395--410.

\bibitem{HH} T. Hirai, A. Hora,  Spin representations of twisted central products of double covering finite groups and the case of permutation groups, \emph{J. Math. Soc. Japan} \textbf{66} (2014), 1191--1226.

\bibitem{Hum} J. Humphreys, On certain projective modular representations of direct products, \emph{J. London Math. Soc.} \textbf{32} (1985), 449--459.

\bibitem{MO90} G. Michler, J. Olsson, The Alperin-McKay conjecture holds in the covering groups of symmetric and alternating groups, $p\neq2$, \emph{J. Reine. Angew. Math.} \textbf{405} (1990), 78--111.

\bibitem{MO91} G. Michler, J. Olsson, Weights for covering groups of symmetric and alternating groups, $p\neq2$, \emph{Can. J. Math.} \textbf{43} (1991), 792--813.

\bibitem{NT} H. Nagao and Y. Tsushima, \emph{Representations of Finite Groups}. Academic Press Inc., Boston, MA, 1989.

\bibitem{Nav} G. Navarro, The McKay conjecture and Galois automorphisms, \emph{Ann. Math.} \textbf{160} (2004), 1129--1140.

\bibitem{Navbook} G. Navarro, \emph{Character Theory and the McKay Conjecture}. Cambridge Stud. Adv. Math., vol. 175, Cambridge University Press, Cambridge, 2018.

\bibitem{Ols92} J. Olsson, The number of modular characters in certain $p$-blocks, \emph{Proc. London Math. Soc.} \textbf{65} (1992), 245--264.

\bibitem{Ols93} J. Olsson, \emph{Combinatorics and Representations of Finite Groups}. Vol. 20. Vorlesungen aus dem Fachbereich Mathematik der Universit\"{a}t GH Essen. Universit\"{a}t Essen, Fachbereich Mathematik, Essen, 1993.

\bibitem{Sch} J. Schur, \"{U}ber die darstellung der symmetrischen und der alternierenden Gruppe durch gebrochene lineare substitutionen, \emph{J. Reine. Angew. Math.} \textbf{139} (1911), 155--250.

\bibitem{Tur} A. Turull, The strengthened Alperin weight conjecture for $p$-solvable groups, \emph{J. Algebra} \textbf{398} (2014), 469--480.
		
	\end{thebibliography}
\end{document}